\documentclass[10pt]{article}
\usepackage{amssymb}
\usepackage{amsmath,amsthm,enumerate,verbatim}
\usepackage{amsfonts}
\usepackage{color}
\newtheorem {theorem}{Theorem}[section]
\newtheorem {proposition}{Proposition}[section]
\newtheorem {corollary}{Corollary}[section]
\newtheorem {lemma}{Lemma}[section]
\newtheorem {assumption}{Assumption}[section]

\newtheorem {definition}{Definition}[section]
\newtheorem {remark}{Remark}[section]
\def\cS{{\mathcal{S}}}

\def\R{{\mathbb{R}}}

\def\argmin{\mathop{\rm arg\,min}}
\def\Argmin{\mathop{\rm Arg\,min}}

\def\ox{{\bar x}}
\def\oy{{\bar y}}

\def\ov{{\bar v}}

\title{\sf Calculus of the exponent of Kurdyka-{\L}ojasiewicz inequality and its applications to linear convergence of first-order methods}

\author{
Guoyin Li \thanks{Department of Applied
Mathematics, University of New South Wales, Sydney 2052, Australia.
E-mail: {g.li@unsw.edu.au}. This author was partially supported by a research grant from Australian Research Council.}
\and Ting Kei Pong \thanks{Department of Applied Mathematics, the Hong Kong Polytechnic University, Hong Kong.
This author was supported partly by Hong Kong Research Grants Council PolyU153085/16p. E-mail: {tk.pong@polyu.edu.hk}.}
}

\date{Revised Version: August 30, 2021}   

\begin{document}
\maketitle

\begin{abstract}
  In this paper, we study the Kurdyka-{\L}ojasiewicz (KL) exponent, an important quantity for analyzing the convergence rate of first-order methods. Specifically, we develop various calculus rules to deduce the KL exponent of new (possibly nonconvex and nonsmooth) functions formed from functions with known KL exponents. In addition, we show that the well-studied Luo-Tseng error bound together with a mild assumption on the separation of stationary values implies that the KL exponent is $\frac12$. The Luo-Tseng error bound is known to hold for a large class of concrete structured optimization problems, and thus we deduce the KL exponent of a large class of functions whose exponents were previously unknown. Building upon this and the calculus rules, we are then able to show that for many convex or nonconvex optimization models for applications such as sparse recovery, their objective function's KL exponent is $\frac12$. This includes the least squares problem with smoothly clipped absolute deviation (SCAD) regularization or minimax concave penalty (MCP) regularization and the logistic regression problem with $\ell_1$ regularization. Since many existing local convergence rate analysis for first-order methods in the nonconvex scenario relies on the KL exponent, our results
  enable us to obtain explicit convergence rate for various first-order methods when they are applied to a large variety of practical optimization models. Finally, we further illustrate how our results can be applied to establishing local linear convergence of the proximal gradient algorithm and the inertial proximal algorithm with constant step-sizes for some specific models that arise in sparse recovery.
\end{abstract}

\section{Introduction}

Large-scale nonsmooth and nonconvex optimization problems are ubiquitous in machine learning and data analysis. Tremendous efforts have thus been directed at designing efficient algorithms for solving these problems. One popular class of algorithms is the class of first-order methods. These methods are noted for their simplicity, ease-of-implementation and relatively (often surprisingly) good performance; some notable examples include the
proximal gradient algorithm, the inertial proximal algorithms and the alternating direction method of multipliers, etc.
Due to the excellent performance and wide applicability of first-order methods, their convergence behaviors have been extensively studied in recent years; see, for example, \cite{AmesHong14,Jerome,AtBoSv13,Chambolle2014,HongLuoRa16,Johnstone2015,LiPong14_2,LiWu95,OCBP14} and references therein. Analyzing the convergence rate of first-order methods is an important step towards a better understanding of existing algorithms, and is also crucial for developing new optimization models and numerical schemes.

As demonstrated in \cite[Theorem~3.4]{Jerome}, the convergence behavior of many first-order methods can be understood using the celebrated Kurdyka-{\L}ojasiewicz (KL) property and its associated KL exponent; see Definitions~\ref{def:KL} and \ref{def:KLexponent}. The KL property and its associated KL exponent have their roots in algebraic geometry, and they describe a qualitative relationship between the value of a suitable potential function (depending on the optimization model and the algorithm being considered) and some first-order information (gradient or subgradient) of the potential function. The KL property has been applied to analyzing local convergence rate of various first-order methods for a wide variety of problems by many researchers; see, for example, \cite{Jerome,FrGaPe15,LiPong14_2,XuYin13}. In these studies, a proto-typical theorem on convergence rate takes the following form:

\noindent
{\bf Prototypical result on convergence rate.} {\it For a certain algorithm of interest, consider a suitable potential function. Suppose that the potential function satisfies the KL property with an exponent of $\alpha\in [0,1)$, and that $\{x^k\}$ is a bounded sequence generated by the algorithm. Then the following results hold.
\begin{enumerate}[{\rm (i)}]
  \item If $\alpha = 0$, then $\{x^k\}$ converges finitely.
  \item If $\alpha\in (0,\frac12]$, then $\{x^k\}$ converges locally linearly.
  \item If $\alpha\in (\frac12,1)$, then $\{x^k\}$ converges locally sublinearly.
\end{enumerate}
}
\noindent
While this kind of convergence results is prominent and theoretically powerful, for the results to be fully informative, one has to be able to estimate the KL exponent. Moreover, in order to guarantee a local linear convergence rate, it is desirable to be able to determine whether a given model has a KL exponent of at most $\frac12$, or be able to construct a new model whose KL exponent is at most $\frac12$ if the old one does not have the desired KL exponent.

However, as noted in \cite[Page 63, Section 2.1]{Luo_Pang_MPEC}, the KL exponent of a given function is often extremely hard to determine or estimate. There are only few results available in the literature concerning explicit KL exponent of a function. One scenario where an explicit estimate of the KL exponent is known is when the function can be expressed as the maximum of finitely many polynomials. In this case, it has been shown in \cite[Theorem~3.3]{Li_Mor_Pham} that the KL exponent can be estimated explicitly in terms of the dimension of the underlying space and the maximum degree of the involved polynomials. However, the derived estimate grows rapidly with the dimension of the problem, and so, leads to rather weak sublinear convergence rate. It is only until recently that a dimension-independent KL exponent of convex piecewise linear-quadratic functions is known, thanks to \cite[Theorem~5]{BNPS15} that connects the KL property and the concept of error bound\footnote{This notion is different from the Luo-Tseng error bound to be discussed in Definition~\ref{def:LTeb}.} for convex functions. In addition, a KL exponent of $\frac12$ is only established in \cite{LWS15} very recently for a class of quadratic optimization problems with matrix variables satisfying orthogonality constraints. Nevertheless, the KL exponent of many common optimization models, such as the least squares problem with smoothly clipped absolute deviation (SCAD) regularization \cite{Fan97} or minimax concave penalty (MCP) regularization \cite{Zhang10} and the logistic regression problem with $\ell_1$ regularization \cite{Shi_Yin_Osher_Sajda}, are still unknown to the best of our knowledge. In this paper, we attempt to further address the problem of determining the explicit KL exponents of optimization models, especially for those that arise in practical applications.

The main contributions of this paper are the rules for computing {\em explicitly} the KL exponent of many (convex or nonconvex) optimization models that arise in applications such as statistical machine learning. We accomplish this via two different means: studying calculus rules and building connections with the concept of Luo-Tseng error bound; see Definition~\ref{def:LTeb}. The Luo-Tseng error bound was used for establishing local linear convergence for various first-order methods, and was shown to hold for a wide range of problems; see, for example, \cite{LuoT92,LuoT92_2,LuoT93,Tse10,TseY09,ZSo15} for details. This concept is different from the error bound studied in \cite{BNPS15} because the Luo-Tseng error bound is defined for specially structured optimization problems and involves first-order information, while the error bound studied in \cite{BNPS15} does not explicitly involve any first-order information. The different nature of these two concepts was also noted in \cite[Section 1]{BNPS15}, in which the Luo-Tseng error bound was referred as ``first-order error bound".

In this paper, we first study various calculus rules for the KL exponent. For example, we deduce the KL exponent of the minimum of finitely many KL functions, the KL exponent of the Moreau envelope of a convex KL function, and the KL exponent of a convex objective from its Lagrangian relaxation, etc., under suitable assumptions. This is the context of Section~\ref{sec3}. These rules are useful in our subsequent analysis of the KL exponent of concrete optimization models that arise in applications. Next, we show that if the Luo-Tseng error bound holds and a mild assumption on the separation of stationary values is satisfied, then the function is a KL function with an exponent of $\frac12$. This is done in Section~\ref{sec:LT}. Upon making this connection, we can now take advantage of the relatively better studied concept of Luo-Tseng error bound, which is known to hold for a wide range of concrete optimization problems; see, for example, \cite{LuoT92,LuoT92_2,LuoT93,Tse10,TseY09,ZSo15}. Hence, in Section~\ref{sec6}, building upon the calculus rules and the connection with Luo-Tseng error bound, we show that many optimization models that arise in applications such as sparse recovery have objectives whose KL exponent is $\frac12$; this covers the least squares problem with
smoothly clipped absolute deviation (SCAD) \cite{Fan97} or minimax concave penalty (MCP) \cite{Zhang10}
 regularization, and the logistic regression problem with $\ell_1$ regularization \cite{Shi_Yin_Osher_Sajda}. In addition, we also illustrate how our result can be used for establishing linear convergence of some first-order methods, such as the proximal gradient algorithm and the inertial proximal algorithm \cite{OCBP14} with constant step-sizes. Finally, we present some concluding remarks in Section~\ref{sec:conclude}.

\section{Notation and preliminaries}

In this paper, we use $\R^n$ to denote the $n$-dimensional Euclidean space, equipped with the standard inner product $\langle\cdot,\cdot\rangle$ and the induced norm $\|\cdot\|$. The closed ball centered at $x \in \R^n$ with radius $r$ is denoted by $B(x,r)$. We denote the nonnegative orthant by $\R^n_+$, and the set of $n\times n$ symmetric matrices by ${\cal S}^n$. For a vector $x\in \R^n$, we use $\|x\|_1$ to denote the $\ell_1$ norm and $\|x\|_0$ to denote the number of entries in $x$ that are nonzero (``$\ell_0$ norm"). For a (nonempty) closed set $D\subseteq \R^n$, the indicator function $\delta_D$ is defined as
\[
\delta_D(x) = \begin{cases}
  0 & {\rm if}\ x\in D,\\
  \infty & {\rm otherwise}.
\end{cases}
\]
In addition, we denote the distance from an $x\in \R^n$ to $D$ by ${\rm dist}(x,D) = \inf_{y\in D}\|x - y\|$, and the set of points in $D$ that achieve this infimum (the projection of $x$ onto $D$) is denoted by ${\rm Proj}_D(x)$. The set ${\rm Proj}_D(x)$ becomes a singleton if $D$ is a closed convex set. Finally, we write ${\rm ri}\,D$ to represent the relative interior of a closed convex set $D$.

For an extended-real-valued function $f:\R^n\rightarrow [-\infty,\infty]$, the domain is defined as ${\rm dom}\,f=\{x:\; f(x) < \infty\}$. Such a function is called proper if it is never $-\infty$ and its domain is nonempty, and is called closed if it is lower semicontinuous.
For a proper function $f:\R^n \to (-\infty,\infty]$, we let $z\stackrel{f}{\to}x$ denote $z\to x$ and $f(z)\to f(x)$. The regular subdifferential of a proper function $f$ \cite[Page 301, Definition~8.3(a)]{Rock98} at $x\in\mathrm{dom}\,f$ is given by
\begin{equation*}
\widehat{\partial} f(x):=\left\{v\in\R^n :\; \liminf_{z\to x, z \neq x}\frac{f(z)-f(x)-\langle v,z-x\rangle}{\|z-x\|}\ge 0 \right\}.
\end{equation*}
The (limiting) {\em subdifferential} of a proper function $f$ \cite[Page 301, Definition~8.3(b)]{Rock98} at $x\in\mathrm{dom}\,f$ is then defined by
\begin{equation}\label{ls}
\partial f(x):=\left\{v\in\R^n :\;\exists x^t\stackrel{f}{\to}x,\;v^t\to v\;\mbox{ with } v^t \in \widehat{\partial} f(x^t)  
\mbox{ for each }t\right\}.
\end{equation}
By convention, if $x \notin {\rm dom} \, f$, then $\partial f(x)=\widehat{\partial} f(x) =\emptyset$.
We also write ${\rm dom}\,\partial f := \{x\in \R^n:\; \partial f(x)\neq \emptyset\}$. It is well known that when $f$ is continuously differentiable,
the subdifferential \eqref{ls} reduces to the gradient of $f$ denoted by $\nabla f$; see, for example, \cite[Exercise~8.8(b)]{Rock98}. Moreover, when $f$ is convex, the subdifferential \eqref{ls} reduces to the classical subdifferential in convex analysis; see, for example, \cite[Proposition~8.12]{Rock98}. The limiting subdifferential enjoys rich and comprehensive calculus rules and has been widely used in nonsmooth and nonconvex optimization \cite{Boris_book,Rock98}.
We also define the limiting (resp. regular) normal cone of a closed set $S$ at $x \in S$ as $N_S(x)=\partial \delta_S(x)$ (resp. $\widehat{N}_S(x)=\widehat{\partial} \delta_S(x)$) where $\delta_S$ is the indicator function of $S$.
A closed set $S$ is called regular at $x \in S$ if $N_S(x) = \widehat{N}_S(x)$ (see \cite[Definition~6.4]{Rock98}), and a proper closed function $f$ is called regular at $x\in {\rm dom}\,f$ if its epigraph ${\rm epi}f$ is
regular at the point $(x,f(x)) \in {\rm epi}\,f$ (see \cite[Definition~7.25]{Rock98}).
Finally, we say that $x$ is a stationary point of proper closed function $f$ if $0\in \partial f(x)$. It is known that any local minimizer of $f$ is a stationary point; see, for example, \cite[Theorem~10.1]{Rock98}.

For a proper closed convex function $P$, the proximal mapping at any $z$ is defined as
  \begin{equation*}
    {\rm prox}_P(z) = \argmin_{x\in \R^n} \left\{P(x) + \frac{1}{2}\|x - z\|^2\right\},
  \end{equation*}
where $\argmin$ denotes the unique minimizer of the optimization problem $\min\limits_x \{P(x) + \frac{1}{2}\|x - z\|^2\}$.\footnote{This problem has a unique minimizer because the objective is proper closed and strongly convex. For a general optimization problem $\min\limits_x f(x)$, we use $\Argmin f$ to denote the set of minimizers, which may be empty, a singleton or may contain more than one point.}
This mapping is nonexpansive, i.e., for any $y$ and $z$, we have
\begin{equation}\label{eq:prox_nonexpansive}
  \|{\rm prox}_P(y) - {\rm prox}_P(z)\| \le \|y - z\|;
\end{equation}
see, for example, \cite[Page~340]{Roc70}.
Moreover, it is routine to show that $x = {\rm prox}_P(z)$ if and only if $z\in x + \partial P(x)$.

The following property is defined for proper closed functions of the form $f = h + P$, where $h$ is a proper closed function with an open domain, and is continuously differentiable with a locally Lipschitz continuous gradient on ${\rm dom}\, h$, and $P$ is proper closed convex. Recall that for this class of functions, we have $\bar x \in {\cal X}$ if and only if $\bar x = {\rm prox}_P(\bar x - \nabla h(\bar x))$, where $\cal X$ denotes the set of stationary points of $f$. Indeed, we have
\[
\bar x\in {\cal X} \Longleftrightarrow 0 \in \partial (h+P)(\bar x) \stackrel{(i)}{\Longleftrightarrow} 0 \in \nabla h(\bar x) + \partial P(\bar x)
\Longleftrightarrow \bar x - \nabla h(\bar x)\in \bar x + \partial P(\bar x) \Longleftrightarrow \bar x = {\rm prox}_P(\bar x - \nabla h(\bar x)),
\]
where (i) follows from \cite[Exercise~8.8(c)]{Rock98}.

\begin{definition}\label{def:LTeb}{\bf (Luo-Tseng error bound)\footnote{We adapt the definition from \cite[Assumption~2a]{TseY09}.}}
Let $\cal X$ be the set of stationary points of $f$.  Suppose that ${\cal X}\neq \emptyset$. We say that the Luo-Tseng error bound \footnote{This is referred as first-order error bound in \cite[Section~1]{BNPS15}.} holds if for any $\zeta \ge \inf f$, there exist $c$, $\epsilon > 0$ so that
  \begin{equation}\label{eq:LTeb}
    {\rm dist}(x,{\cal X}) \le c \|{\rm prox}_P(x - \nabla h(x)) - x\|
  \end{equation}
  whenever $\|{\rm prox}_P(x - \nabla h(x)) - x\| < \epsilon$ and $f(x) \le \zeta$.
\end{definition}
It is known that this property is satisfied for many choices of $h$ and $P$, and we refer to \cite{LuoT92,LuoT92_2,LuoT93,Tse10,TseY09,ZSo15} and references therein for more detailed discussions. This property was used for establishing local linear convergence of various first-order methods applied to minimizing $f = h + P$.

Recently, the following property was also used extensively for analyzing convergence rate of first-order methods, mainly for possibly nonconvex objective functions; see, for example, \cite{Jerome,AtBoSv13}.

\begin{definition}\label{def:KL} {\bf (KL property \& KL function)}
  We say that a proper closed function $f$ has the Kurdyka-{\L}ojasiewicz (KL) property at $\bar{x}\in {\rm dom}\,\partial f$ if there exist a neighborhood $\cal N$ of $\bar{x}$, $\nu\in (0,\infty]$ and a continuous concave function $\psi:[0,\nu)\rightarrow {\mathbb{R}}_+$ with $\psi(0) = 0$ such that:
  \begin{enumerate}[{\rm (i)}]
    \item $\psi$ is continuously differentiable on $(0,\nu)$ with $\psi' > 0$;
    \item for all $x\in {\cal N}$ with $f(\bar{x})< f(x) < f(\bar{x}) + \nu$, one has
    \begin{equation*}
      \psi'(f(x) - f(\bar{x}))\,{\rm dist}(0,\partial f(x))\ge 1.
    \end{equation*}
  \end{enumerate}
  A proper closed function $f$ satisfying the KL property at all points in ${\rm dom}\,\partial f$ is called a KL function.
\end{definition}

In this paper, we are interested in the KL exponent, which is defined \cite{Jerome,AtBoSv13} as follows.
\begin{definition}{\bf (KL exponent)}\label{def:KLexponent}
  For a proper closed function $f$ satisfying the KL property at $\bar x\in {\rm dom}\,\partial f$, if the corresponding function $\psi$ can be chosen as $\psi(s)= c\, s^{1-\alpha}$ for some $c>0$ and $\alpha \in [0,1)$, i.e.,
  there exist $c$, $\epsilon > 0$ and $\nu \in (0,\infty]$ so that
  \begin{equation}\label{eq:00}
    {\rm dist}(0,\partial f(x)) \ge c(f(x) - f(\bar x))^\alpha
  \end{equation}
  whenever $\|x - \bar x\| \le \epsilon$ and $f(\bar x)< f(x) < f(\bar x) + \nu$,
  then we say that $f$ has the KL property at $\bar x$ with an exponent of $\alpha$. If $f$ is a KL function and has the same exponent $\alpha$ at any $\bar x\in {\rm dom}\,\partial f$, then we say that $f$ is a KL function with an exponent of $\alpha$.  \footnote{In classical algebraic geometry, the exponent $\alpha$ is also referred as the {\L}ojasiewicz exponent.}
\end{definition}

This definition encompasses broad classes of function that arise in practical optimization problems. For example, it is known that if $f$ is a proper closed semi-algebraic function \cite{AtBoSv13}, then $f$ is a KL function with a suitable exponent $\alpha \in [0,1)$. As established in \cite[Theorem~3.4]{Jerome} and many subsequent work, KL exponent has a close relationship with the rate of convergence of many commonly used optimization methods.

Before ending this section, we state two auxiliary lemmas. The first result is an immediate consequence of the fact that the set-valued mapping $x\mapsto \partial f(x)$ is outer semicontinuous (with respect to the $f$-attentive convergence, i.e., $y\stackrel{f}{\to}x$; see \cite[Proposition~8.7]{Rock98}), and can be found in \cite[Remark 4 (b)]{Jerome}. This result will be used repeatedly at various places in our discussion below. We include a proof for self-containedness.

\begin{lemma}\label{lem:prep2}
Suppose that $f$ is a proper closed function, $\bar x\in {\rm dom}\,\partial f$ and $0\notin \partial f(\bar x)$. Then,  for any $\alpha\in [0,1)$, $f$ satisfies the KL property at $\bar x$ with an exponent of $\alpha$.
\end{lemma}

\begin{proof}
  Fix any $\alpha\in [0,1)$.
  Since $0\notin \partial f(\bar x)$ and $\partial f(\bar x)$ is nonempty and closed, it follows that ${\rm dist}(0,\partial f(\bar x))$ is positive and finite. Define $\eta := \frac12 {\rm dist}(0,\partial f(\bar x)) > 0$. We claim that there exists $\epsilon \in (0,1)$ so that ${\rm dist}(0,\partial f(x)) > \eta$ whenever $\|x-\bar x\| \le \epsilon$ and $f(\bar x)< f(x) < f(\bar x) + \epsilon$.

  Suppose for the sake of contradiction that this is not true. Then there exists a sequence $\{x^k\}$ with $x^k \to \bar x$ and $f(x^k)\to f(\bar x)$ so that
  \[
  {\rm dist}(0,\partial f(x^k)) \le \eta.
  \]
  In particular, there exists a sequence $\{\xi^k\}$ satisfying $\xi^k\in \partial f(x^k)$ and $\|\xi^k\|\le \eta$. By passing to a subsequence if necessary, we may assume without loss of generality that $\xi^k \to \bar \xi$ for some $\bar\xi$, and we have $\bar\xi\in \partial f(\bar x)$, thanks to \cite[Proposition~8.7]{Rock98}. But then we have $2\eta = {\rm dist}(0,\partial f(\bar x)) \le \|\bar \xi\| \le \eta$, a contradiction. Thus, there exists $\epsilon \in (0,1)$ so that ${\rm dist}(0,\partial f(x)) > \eta$ whenever $\|x-\bar x\| \le \epsilon$ and $f(\bar x)< f(x) < f(\bar x) + \epsilon$.

  Using this, we see immediately that
  \[
  \eta(f(x) - f(\bar x))^\alpha \le \eta < {\rm dist}(0,\partial f(x))
  \]
  whenever $\|x-\bar x\| \le \epsilon$ and $f(\bar x)< f(x) < f(\bar x) + \epsilon$, showing that $f$ satisfies the KL property at $\bar x$ with an exponent of $\alpha$. This completes the proof.
\end{proof}

The second result concerns the equivalence of ``norms", whose proof is simple and is omitted.

\begin{lemma}\label{ainequ}
  Let $t > 0$. Then there exist $C_1\ge C_2 > 0$ so that
  \[
  C_2\|a\| \le (a_1^t + \cdots + a_m^t)^\frac1t \le C_1\|a\|
  \]
  for any $a = (a_1,\ldots,a_m)\in \R^m_+$.
\end{lemma}

\section{Calculus of the KL exponent}\label{sec3}

In this section, we discuss how the KL exponent behaves under various operations on KL functions. We briefly summarize our results below. The required assumptions will be made explicit in the respective theorems.
\begin{enumerate}[{\rm (i)}]
  \item Exponent for $\min_{1\le i\le r}f_i$ given the exponents of $f_i$ for each $i$; see Theorem~\ref{thm:minKL} and Corollary~\ref{cor:minKL}.
  \item Exponent for $g(F(x))$ when the Jacobian of $F$ is surjective, given the exponent of $g$; see Theorem~\ref{thm:compositeKL}.
  \item Exponent for $\sum_{i=1}^mf_i(x_i)$ given the exponents of $f_i$ for each $i$; see Theorem~\ref{thm:block_separable_KL}.
  \item Exponent for the Moreau envelope of a convex KL function; see Theorem~\ref{thm:Moreau}.
  \item Deducing the exponent from the Lagrangian relaxation for convex problems; see Theorem~\ref{thm:Lagrange}.
  \item Exponent for a potential function used in the convergence analysis of the inertial proximal algorithm in \cite{OCBP14}; see Theorem~\ref{thm:extra}.
    \item Deducing the exponent of a partly smooth KL function by looking at its restriction on its active manifold; see Theorem~\ref{thm:partlysmooth}.
\end{enumerate}
We shall make use of some of these calculus rules in Section~\ref{sec6} to deduce the KL exponent of some concrete optimization models.

We start with our first result, which concerns the minimum of finitely many KL functions. This rule will prove to be useful in Section~\ref{sec6}. Indeed, as we shall see there, many nonconvex optimization problems that arise in applications have objectives that can be written as the minimum of finitely many KL functions whose exponents can be deduced from our results in Section~\ref{sec:LT}; this includes some prominent and widely used NP-hard optimization model problems, for example, the least squares problem with cardinality constraint \cite{BluDav08}.

\begin{theorem} {\bf (Exponent for minimum of finitely many KL functions)} \label{thm:minKL}
  Let $f_i$, $1\le i\le r$, be proper closed functions, $f := \min_{1\le i\le r}f_i$ be continuous on ${\rm dom}\,\partial f$ and $\bar x\in {\rm dom}\,\partial f\cap \big(\bigcap_{i\in I(\bar x)}{\rm dom}\,\partial f_i\big)$, where $I(\bar{x}):=\{i: f_i(\bar{x})=f(\bar{x})\}$. Suppose further that each $f_i$, $i\in I(\bar x)$, satisfies the KL property at $\bar x$ with an exponent of $\alpha_i \in [0,1)$. Then $f$ satisfies the KL property at $\bar x$ with an exponent of
$\alpha=\max\{\alpha_i: i\in I(\bar x)\}$.
\end{theorem}
\begin{proof}
From the definition of $I(\bar x)$, we see that
$\min_{i \notin I(\bar{x})}\{f_i(\bar{x})\}> f(\bar{x})$.
Since $x\mapsto \min_{i \notin I(\bar{x})}\{f_i(x)\}$ is lower semicontinuous and the function $f$ is continuous on ${\rm dom}\,\partial f$,
there exists $\eta > 0$ such that for all $x \in {\rm dom}\,\partial f$ with $\|x - \bar x\|\le \eta$, we have
\[
\min_{i \notin I(\bar{x})}\{f_i(x)\} >  f(x).
\]
Thus, whenever $x \in {\rm dom}\,\partial f$ and $\|x - \bar x\|\le \eta$, we have $I(x)\subseteq I(\bar x)$.

Next, using $f(x)=\min_{1 \le i \le r}\{f_i(x)\}$ and the subdifferential rule of the minimum of finitely many functions \cite[Theorem~5.5]{MoS95}, we obtain for all $x\in {\rm dom}\,\partial f$ that
\begin{equation}\label{subinclu}
\partial f(x) \subseteq \bigcup_{i \in I(x)} \partial f_i(x).
\end{equation}
On the other hand, by assumption, for each $i\in I(\bar x)$, there exist $\epsilon_i$, $c_i$, $\nu_i>0$ such that for all $x \in {\rm dom}\,\partial f_i$ with $\|x-\bar{x}\|\le \epsilon_i$ and $f_i(\bar x)< f_i(x)< f_i(\bar x)+\nu_i$, one has
\begin{equation}\label{KL}
{\rm dist}(0,\partial f_i(x)) \ge c_i    \big(f_i(x)-f_i(\bar x)\big)^{\alpha_i}.
\end{equation}
Let $c=\min_{i\in I(\bar x)}c_i$, $\epsilon=\min\{\eta,\min_{i\in I(\bar x)}\epsilon_i\}$ and $\nu=\min_{i\in I(\bar x)}\{1,\nu_i\}$. Take any $x \in {\rm dom}\,\partial f$ with $\|x-\bar{x}\|\le \epsilon$ and $f(\bar x)< f(x)< f(\bar x)+\nu$. Then $I(x)\subseteq I(\bar x)$ and we have
\[
{\rm dist}(0,\partial f(x)) \ge  \min_{i\in I(x)}{\rm dist}(0,\partial f_i(x)) \ge c   \min_{i\in I(x)}\big(f_i(x)-f_i(\bar x)\big)^{\alpha} = c \big(f(x)-f(\bar x)\big)^{\alpha},
\]
where the first inequality follows from \eqref{subinclu}, the second inequality follows from \eqref{KL}, the construction of $c$, $\nu$, $\epsilon$ and $\alpha$, as well as the facts that $f_i(x)=f(x)$ and $f_i(\bar x) = f(\bar x)$ for $i\in I(x)\subseteq I(\bar x)$; these facts also give the last equality. This completes the proof.
\end{proof}

We have the following immediate corollary.
\begin{corollary}\label{cor:minKL}
  Let $f_i$, $1\le i\le r$, be proper closed functions with ${\rm dom}\,f_i = {\rm dom}\,\partial f_i$ for all $i$, and $f := \min_{1\le i\le r}f_i$ be continuous on ${\rm dom}\,\partial f$. Suppose further that each $f_i$ is a KL function with an exponent of $\alpha_i \in [0,1)$ for $1\le i\le r$. Then $f$ is a KL function with an exponent of $\alpha=\max\{\alpha_i: 1\le i\le r\}$.
\end{corollary}
\begin{proof}
  In view of Theorem~\ref{thm:minKL}, it suffices to show that for any $\bar x\in {\rm dom}\,\partial f$, we have $\bar x\in \bigcap_{i\in I(\bar x)}{\rm dom}\,\partial f_i$.
  To this end, take any $\bar x\in {\rm dom}\,\partial f$. Note that we have ${\rm dom}\,\partial f\subseteq {\rm dom}\, f$ by the definition, and hence $f(\bar x) < \infty$. In addition, from the definition of $I(\bar x)$, we have $f(\bar x) = f_i(\bar x)$ for all $i\in I(\bar x)$. Hence, $f_i(\bar x)$ is finite for all $i \in I(\bar x)$. Thus, we conclude that $\bar x\in {\rm dom}\,f_i$ for all $i \in I(\bar x)$, which implies that $\bar x\in \bigcap_{i\in I(\bar x)}{\rm dom}\,\partial f_i$ because ${\rm dom}\,f_i = {\rm dom}\,\partial f_i$ for all $i$ by assumption.
\end{proof}

The next theorem concerns the composition of a KL function with a smooth function that has a surjective Jacobian mapping.

\begin{theorem} {\bf (Exponent for composition of KL functions)} \label{thm:compositeKL}
  Let $f(x) = (g\circ F)(x)$, where $g$ is a proper closed function on $\mathbb{R}^{m}$ and $F:\mathbb{R}^n \rightarrow \mathbb{R}^m$ is a continuously differentiable mapping. Suppose in addition that $g$ is a KL function with an exponent of $\alpha \in [0,1)$ and the Jacobian $J F(\bar x)\in \R^{m\times n}$ is a surjective mapping at some $\bar x\in {\rm dom}\,\partial f$. Then $f$ has the KL property at $\bar x$ with an exponent of $\alpha$.
\end{theorem}
\begin{proof}
Note from \cite[Exercise~10.7]{Rock98} and $\bar x\in {\rm dom}\,\partial f$ that $F(\bar x)\in {\rm dom}\,\partial g$. As $g$ is a KL function, there exist $\epsilon_0$, $c_0$, $\nu_0>0$ such that for all $u \in {\rm dom}\,\partial g$ with $\|u-F(\bar{x})\|\le \epsilon_0$ and $g(F(\bar x))< g(u)< g(F(\bar x))+\nu_0$, one has
\begin{equation}\label{eq:1001}
{\rm dist}(0,\partial g(u)) \ge  c_0   \big(g(u)-g(F(\bar x))\big)^{\alpha}.
\end{equation}
On the other hand, since the linear map $J F( \bar x)$ is surjective and $F$ is continuously differentiable, it follows from the classical Lyusternik-Graves theorem (see, for example, \cite[Theorem 1.57]{Boris_book})
that there are numbers $l> 0$ and $\epsilon_1 \in (0,\epsilon_0)$ such that for all $x$ with $\|x -\bar x\|\le \epsilon_1$
\[
l \, \mathbb{B}_{\mathbb{R}^m} \subseteq J F(x)(\mathbb{B}_{\mathbb{R}^n}),
\]
where $\mathbb{B}_{\mathbb{R}^m}$ and $\mathbb{B}_{\mathbb{R}^n}$ are the closed unit balls in $\mathbb{R}^m$ and $\mathbb{R}^n$, respectively.
This implies that for all $z \in \mathbb{R}^m$ we have the following estimate:
\begin{equation}\label{eq:100}
\|J F(x)^* z\| \ge l \|z\|,
\end{equation}
whenever $\|x -\bar x\|\le \epsilon_1$. Moreover, from the chain rule of the limiting subdifferential for composite functions (see, for example, \cite[Exercise~10.7]{Rock98}),  we have  for all $x$ with $\|x-\bar x\| \le \epsilon_1$ that
\[
\partial f(x) = \{J F(x)^* v: v \in  \partial g(F(x))\},
\]
because $JF(x)$ is a surjective mapping for all such $x$.

Now, let $\epsilon \in (0,\epsilon_1]$ be such that $\|F(x)-F(\bar{x})\| \le \epsilon_0$ for all $\|x-\bar{x}\|\le \epsilon$, $c= l \, c_0$ and $\nu=\nu_0$.
Fix $x \in {\rm dom}\,\partial f$ with $\|x-\bar{x}\|\le \epsilon$ and $f(\bar x)< f(x)< f(\bar x)+\nu$. Let $a \in \partial f(x)$ be such that
$\|a\|={\rm dist}(0,\partial f(x))$.
Then, we have $a=J F(x)^* v$ for some  $v \in  \partial g(F(x))$. Hence, it follows from \eqref{eq:100} that
\[
\|a\|=\|J F(x)^* v\| \ge l \|v\|.
\]
In addition, since $v \in \partial g(F(x))$, applying \eqref{eq:1001} with $u=F(x)$ gives us that
\[
\|v\| \ge {\rm dist}\big(0, \partial g(F(x))\big) \ge c_0   \big(g(F(x))-g(F(\bar x))\big)^{\alpha}= c_0 \big(f(x)-f(\bar x)\big)^{\alpha}.
\]
Therefore,
\[
{\rm dist}(0,\partial f(x)) =\|a\| \ge l \, \|v\| \ge c\, \big(f(x)-f(\bar x)\big)^{\alpha}.
\]

\end{proof}

Our next theorem concerns separable sums.

\begin{theorem} {\bf (Exponent for block separable sums of KL functions)} \label{thm:block_separable_KL}
 Let $n_i,n \in \mathbb{N}$, $i=1,\ldots,m$ be such that $\sum_{i=1}^m n_i=n$. Let $f(x) = \sum_{i=1}^m f_i(x_i)$, where $f_i$, $1\le i\le r$, are proper closed functions on $\mathbb{R}^{n_i}$ with $x=(x_1,\ldots,x_m) \in \mathbb{R}^n$. Suppose further that each $f_i$ is a KL function with an exponent of $\alpha_i \in (0,1)$ and that each $f_i$ is continuous on ${\rm dom}\,\partial f_i$, $i=1,\ldots,r$. Then $f$ is a KL function with an exponent of
$\alpha= \max\{\alpha_i:1\le i\le r\}$.
\end{theorem}
\begin{proof}
Denote $\bar x=(\bar x_1,\ldots,\bar x_m)\in {\rm dom}\,\partial f$ with $\bar x_i \in\mathbb{R}^{n_i}$, $i=1,\ldots,m$. Then \cite[Proposition~10.5]{Rock98} shows that $\bar x_i\in {\rm dom}\,\partial f_i$ for each $i$.
As each $f_i$, $i=1,\ldots,m$, is a KL function with exponent $\alpha_i$, there exist $\epsilon_i$, $c_i$, $\nu_i>0$ such that for all $x_i \in {\rm dom}\,\partial f_i \subseteq \mathbb{R}^{n_i}$ with $\|x_i-\bar{x}_i\|\le \epsilon_i$ and $f_i(\bar x_i)< f_i(x_i)< f_i(\bar x_i)+\nu_i$, one has
\begin{equation}\label{KL1}
c_i{\rm dist}^\frac{1}{\alpha_i} (0,\partial f_i(x_i)) \ge f_i(x_i)-f_i(\bar x_i).
\end{equation}
Since the left hand side of \eqref{KL1} is always nonnegative, the above relation holds trivially whenever $f_i(x_i) \le f_i(\bar x_i)$. In addition, since $f_i$ is continuous on ${\rm dom}\,\partial f_i$ by assumption for each $i=1,\ldots,r$, by shrinking $\epsilon_i > 0$ if necessary, we conclude that $f_i(x_i)< f_i(\bar x_i)+\nu_i$ whenever $x_i \in {\rm dom} \partial f_i$  and $\|x_i-\bar{x}_i\|\le \epsilon_i$. Thus, we have from these two observations that for all $x_i \in {\rm dom}\,\partial f_i \subseteq \mathbb{R}^{n_i}$ with $\|x_i-\bar{x}_i\|\le \epsilon_i$,
\begin{equation}\label{KL11}
c_i {\rm dist}^{\frac{1}{\alpha_i}} (0,\partial f_i(x_i)) \ge   f_i(x_i)-f_i(\bar x_i).
\end{equation}

Let $\epsilon=\min_{1 \le i \le r}\epsilon_i$. Take any $x \in {\rm dom}\,\partial f$ with $\|x-\bar{x}\|\le \epsilon$ and $f(\bar x)< f(x)< f(\bar x)+1$. We will now verify \eqref{eq:00}.
To this end, let
$z \in \partial f(x)$ be such that
\[
\|z\|={\rm dist}(0,\partial f(x)).
\]
If $\|z\| > 1$, then clearly
\[
{\rm dist}(0,\partial f(x)) = \|z\| \ge (f(x) - f(\bar x))^\alpha
\]
since $0< f(x) - f(\bar x)< 1$. Thus, we consider the case where $\|z\|\le 1$. In this case, recall from \cite[Proposition~10.5]{Rock98} that
\[
\partial f(x) = \prod_{i=1}^m \partial f_i(x_i).
\]
Hence, there exist $z_i \in \mathbb{R}^{n_i}$ with $z_i \in \partial f_i(x_i)  \mbox{ and } \|z_i\|={\rm dist}(0,\partial f_i(x_i))$ such that
\[
\sum_{i=1}^m \|z_i\|^2=\|z\|^2.
\]
This together with \eqref{KL11} implies that
\begin{equation}\label{ineq3}
C_0\sum_{i=1}^m\|z_i\|^{\frac{1}{\alpha_i}}\ge \sum_{i=1}^m c_i\|z_i\|^{\frac{1}{\alpha_i}} = \sum_{i=1}^m c_i {\rm dist}^{\frac{1}{\alpha_i}} (0,\partial f_i(x_i)) \ge \sum_{i=1}^m  [ f_i(x_i)-f_i(\bar x_i) ] =f(x)-f(\bar x)
\end{equation}
for $C_0 = \max_{1\le i\le r}c_i$.
Define $\alpha:=\max\{\alpha_i: 1 \le i \le r\}$. Since $\|z_i\| \le \|z\| \le 1$ and $0 < \alpha_i \le \alpha <1$, it then follows from \eqref{ineq3} that
\[
(f(x)-f(\bar x))^\alpha\le \left(C_0 \sum_{i=1}^m \, \|z_i\|^{\frac{1}{\alpha}}\right)^\alpha \le C_1C_0^\alpha\|z\| = C_1C_0^\alpha {\rm dist}(0,\partial f(x)),
\]
where the second inequality follows from Lemma~\ref{ainequ} with $t = \frac{1}{\alpha}$. This completes the proof.
\end{proof}

We now discuss the operation of taking Moreau envelope. This operation is a common operation for smoothing the objective function of convex optimization problems.

\begin{theorem}{\bf (Exponent for Moreau envelope of convex KL functions)}\label{thm:Moreau}
  Let $f$ be a proper closed convex function that is a KL function with an exponent of $\alpha\in (0,\frac23)$. Suppose further that $f$ is continuous on ${\rm dom}\,\partial f$. Fix $\lambda > 0$ and consider
  \[
  F_\lambda(x) := \inf_{y\in \R^n}\left\{f(y) + \frac{1}{2\lambda}\|y - x\|^2\right\}.
  \]
  Then $F_\lambda$ is a KL function with an exponent of $\max\{\frac12,\frac{\alpha}{2-2\alpha}\} < 1$.
\end{theorem}

\begin{proof}
  It suffices to consider the case $\Argmin F_\lambda\neq \emptyset$ and show that $F_\lambda$ has the KL property with an exponent of $\max\{\frac12,\frac{\alpha}{2-2\alpha}\}$ at any fixed $\bar x \in \Argmin F_\lambda$, in view of Lemma~\ref{lem:prep2} and the convexity of $F_\lambda$.

  To this end, recall from \cite[Proposition~12.29]{BauCom10} that, for all $x \in \mathbb{R}^n$,
  \begin{equation}\label{grad}
  \nabla F_\lambda(x) = \frac1\lambda (x - {\rm prox}_{\lambda f}(x)),
  \end{equation}
  and that $\nabla F_\lambda$ is Lipschitz continuous with a Lipschitz constant of $\frac1\lambda$. Consequently, we have for any $y$ that
  \begin{equation}\label{grad2}
  \begin{split}
  F_\lambda(y) - F_\lambda(\bar x) &= F_\lambda(y) - F_\lambda(\bar y) \le \langle\nabla F_\lambda(\bar y),y - \bar y\rangle + \frac{1}{2\lambda}\|y - \bar y\|^2 = \frac{1}{2\lambda}\|y - \bar y\|^2,
  \end{split}
  \end{equation}
  where $\bar y$ is the projection of $y$ onto $\Argmin F_\lambda$, and the last equality holds because $\nabla F_\lambda(\bar y) = 0$.

  Next, note that we have $\Argmin f=\{x \in \mathbb{R}^n: x={\rm prox}_f(x)\}$ according to \cite[Proposition~12.28]{BauCom10}, which implies
  $\Argmin f=\{x \in \mathbb{R}^n: x={\rm prox}_{\lambda f}(x)\}$ as $\lambda>0$. This together with \eqref{grad} gives $\Argmin F_\lambda = \Argmin f$. Hence, $\bar x\in \Argmin F_\lambda = \Argmin f\subseteq {\rm dom}\,\partial f$.
  Since $f$ is a KL function with an exponent of $\alpha$ and $\bar x\in {\rm dom}\,\partial f$, using the fact that $f$ is continuous on ${\rm dom}\,\partial f$, we obtain that
  there exist $c_0 > 0$ and $\epsilon > 0$ so that
  \begin{equation}\label{grad3}
    {\rm dist}(0,\partial f(y)) \ge c_0\, (f(y) - f(\bar x))^\alpha = c_0\, (f(y) - \inf f)^\alpha
  \end{equation}
  whenever $\|y - \bar x\| \le \epsilon$ and $y\in {\rm dom}\,\partial f$; here, the condition on the bound on function values is waived by using the continuity of $f$ on $\partial f$ and choosing a smaller $\epsilon$ if necessary. Moreover, in view of \cite[Theorem~5(i)]{BNPS15}, by shrinking $\epsilon$ if necessary, we conclude that there exists $c>0$ so that
  \begin{equation}\label{grad30}
  {\rm dist}(y,\Argmin f) \le c\, (f(y) - \inf f)^{1-\alpha},
  \end{equation}
  whenever $\|y - \bar x\| \le \epsilon$ and $y\in {\rm dom}\,\partial f$; here, the condition on the bound on function values is waived similarly as before.
  Finally, since $\Argmin F_\lambda = \Argmin f$, we have $\|y - \bar y\| = {\rm dist}(y,\Argmin f)$. Combining this with \eqref{grad3} and \eqref{grad30} implies that for some $c_1 > 0$,
  \begin{equation}\label{grad4}
    \|y - \bar y\| \le c_1[{\rm dist}(0,\partial f(y))]^{\frac{1 - \alpha}{\alpha}}
  \end{equation}
  whenever $\|y - \bar x\| \le \epsilon$ and $y\in {\rm dom}\,\partial f$.

  Now, using the definition of the proximal mapping as minimizer, we have by using the first-order optimality condition that for any $u$,
  \[
  0\in \lambda\partial f({\rm prox}_{\lambda f}(u)) + ({\rm prox}_{\lambda f}(u)-u).
  \]
  In particular, ${\rm prox}_{\lambda f}(u)\in {\rm dom}\,\partial f$. In addition, using the above relation and \eqref{grad}, we deduce that
  \begin{equation}\label{grad5_5}
  {\rm dist}(0,\partial f({\rm prox}_{\lambda f}(u))) \le \frac1\lambda\|u - {\rm prox}_{\lambda f}(u)\| = \|\nabla F_\lambda(u)\|.
  \end{equation}
  Fix an arbitrary $u$ with $\|u-\bar x\| \le \epsilon$. Then $\|{\rm prox}_{\lambda f}(u) - \bar x\| = \|{\rm prox}_{\lambda f}(u) - {\rm prox}_{\lambda f}(\bar x)\|\le \|u - \bar x\|$, where the inequality is due to \eqref{eq:prox_nonexpansive}. Let $y = {\rm prox}_{\lambda f}(u)$. Then $y\in {\rm dom}\,\partial f$ and $\|y-\bar x\|\le \epsilon$. Hence, the relations \eqref{grad4} and \eqref{grad5_5} imply that
\begin{equation*}      \|y - \bar y\| \le c_1\, \|\nabla F_\lambda(u)\|^{\frac{1-\alpha}{\alpha}}.
  \end{equation*}
%
  Applying  \eqref{grad2} with $y = {\rm prox}_{\lambda f}(u)$ and combining this with the preceding relation, we obtain further that
  \begin{equation}\label{grad6}
    F_\lambda({\rm prox}_{\lambda f}(u)) - F_\lambda(\bar x)\le \frac{c_1^2}{2\lambda}\|\nabla F_\lambda(u)\|^{\frac{2-2\alpha}{\alpha}}.
  \end{equation}
  whenever $\|u - \bar x\| \le \epsilon$. Finally, from the convexity of $F_\lambda$, we have
  \begin{equation}\label{grad7}
  F_\lambda(u) - F_\lambda({\rm prox}_{\lambda f}(u)) \le \langle\nabla F_\lambda(u),u-{\rm prox}_{\lambda f}(u)\rangle = \lambda\|\nabla F_\lambda(u)\|^2,
  \end{equation}
  where the equality follows from \eqref{grad}.

  Shrink $\epsilon$ further if necessary so that $\|\nabla F_\lambda(u)\| < 1$ whenever $\|u - \bar x\| \le \epsilon$; this is possible since $\nabla F_\lambda(\bar x) = 0$. Summing \eqref{grad6} and \eqref{grad7}, we obtain further that
  \[
  \begin{split}
  F_\lambda(u) - F_\lambda(\bar x) & \le \frac{c_1^2}{2\lambda}\|\nabla F_\lambda(u)\|^{\frac{2-2\alpha}{\alpha}} + \lambda\|\nabla F_\lambda(u)\|^2 \\
  & \le C\|\nabla F_\lambda(u)\|^{\min \left\{2,\frac{2-2\alpha}{\alpha}\right\}}
  \end{split}
  \]
  for some $C > 0$, whenever $\|u - \bar x\| \le \epsilon$. This completes the proof.
\end{proof}

Our next result concerns the Lagrangian relaxation. This result will be used later in Proposition \ref{prop:group_lasso} for a group LASSO model. In addition, the result, together with results in \cite{Li_Mor_Pham} that give the exponent of the maximum of finitely many convex polynomials, can be used to study the KL property of a large class of convex polynomial optimization problems with multiple convex polynomial constraints.

\begin{theorem}{\bf (KL exponent from Lagrangian relaxation)}\label{thm:Lagrange}
  Let $h(x) = l(Ax)$ for some continuous strictly convex function $l:\R^m\to \R$ and $A\in \R^{m\times n}$, $g:\R^n\to \R$ be a continuous convex function and $\alpha\in (0,1)$, $D$ be a closed convex set. Let $C=\{x: g(x) \le 0\}$. Suppose in addition that
  \begin{enumerate}[{\rm (i)}]
    \item there exists $x_0\in D$ with $g(x_0) < 0$;
    \item $\inf\limits_{x\in D} h(x) < \inf\limits_{x\in C \cap D}h(x)$;
    \item for any $\lambda > 0$, the function $h(x) + \lambda g(x) + \delta_D(x)$ is a KL function with an exponent of $\alpha$.
  \end{enumerate}
  Then $f(x) = h(x) + \delta_{C}(x) + \delta_D(x)$ has the KL property at any $\bar x\in C \cap D$, with an exponent of $\alpha$.
\end{theorem}

\begin{proof}
  In view of Lemma~\ref{lem:prep2} and the convexity of $f$, we only need to consider the case $\Argmin f\neq \emptyset$ and look at those $\bar x$ with $0 \in \partial f(\bar x)$ (and so, $\bar x \in \Argmin f$ by the convexity of $f$). From now on, fix any such $\bar x$.

  First, from condition (i) and \cite[Corollary~28.2.1]{Roc70}, there exists $\lambda \ge 0$ so that
  \begin{equation}\label{eq:La1}
  f(\bar x) = \inf_{x\in \R^n} f = h(\bar x) = \inf_{x\in C\cap D} h(x) = \inf_{x\in D} \{h(x) + \lambda g(x)\} \le h(\bar x) + \lambda g(\bar x) \le h(\bar x),
  \end{equation}
  where the first inequality follows from $\bar x\in C\cap D$, while the second inequality follows from the fact that $\bar x\in C$ and hence $g(\bar x)\le 0$. Thus, equality holds throughout \eqref{eq:La1}; in particular, we have $h(\bar x) + \lambda g(\bar x) = h(\bar x)$, which gives $\lambda g(\bar x) = 0$. Also, in view of the fourth equality in \eqref{eq:La1} and condition (ii), we must have $\lambda > 0$; consequently, we have $g(\bar x) = 0$. Fix any such $\lambda > 0$. Then, in view of \cite[Theorem~28.1]{Roc70}, we also have
  \begin{equation}\label{barxin}
  \bar x \in \Argmin f = \{x:\;g(x)=0\}\cap \Argmin (h + \lambda g + \delta_D).
  \end{equation}

  Next, since $h(x) = l(Ax)$ for some strictly convex function $l$, it must hold true that $Ax$ is constant for any $x\in \Argmin (h + \lambda g + \delta_D)$. Since $\lambda > 0$, we deduce that $g(x)$ is constant over $\Argmin (h + \lambda g + \delta_D)$. Hence, in view of \eqref{barxin}, we have $g(x) = g(\bar x) = 0$ for any $x\in \Argmin (h + \lambda g + \delta_D)$. Then we conclude further from \eqref{barxin} that
  \begin{equation}\label{equation4}
  \Argmin f = \{x:\;g(x)=0\}\cap \Argmin (h + \lambda g + \delta_D) = \Argmin (h + \lambda g + \delta_D).
  \end{equation}

  Now, using condition (iii) and \cite[Theorem~5(i)]{BNPS15}, and noting that $\bar x\in {\rm dom}\,\partial(h+\lambda g + \delta_D)$, we see that there exist $\epsilon$, $\nu$ and $c > 0$ so that
  \begin{equation}\label{KL:Lagrange}
    {\rm dist}(x,\Argmin (h + \lambda g + \delta_D)) \le c (h(x) + \lambda g(x) - h(\bar x))^{1- \alpha}
  \end{equation}
  whenever $\|x - \bar x\|\le \epsilon$, $x\in D$ and $h(\bar x) \le h(x) + \lambda g(x) < h(\bar x) + \nu$, because $g(\bar x) = 0$. On the other hand, whenever $h(\bar x) < h(x) < h(\bar x) + \nu$ and $x\in C\cap D$, we have
  \[
  h(\bar x) = \inf_{x\in D}\{h(x) + \lambda g(x)\} \le h(x) + \lambda g(x) \le h(x) < h(\bar x) + \nu,
  \]
  where the equality follows from \eqref{eq:La1} and the second inequality follows from the definition of $C$. Combining this with \eqref{KL:Lagrange}, we see that whenever $x\in C\cap D$, $\|x - \bar x\|\le \epsilon$ and $h(\bar x) < h(x) < h(\bar x) + \nu$, we have
  \begin{equation*}
    \begin{split}
      {\rm dist}(x,\Argmin f) & =  {\rm dist}(x,\Argmin (h + \lambda g + \delta_D)) \le c (h(x) + \lambda g(x) - h(\bar x))^{1- \alpha} \\
      &\le c(h(x) - h(\bar x))^{1- \alpha} =  c(f(x) - f(\bar x))^{1- \alpha},
    \end{split}
  \end{equation*}
  where the first equality follows from \eqref{equation4} and the second inequality follows from the definition of $C$. The conclusion of the theorem now follows from this last relation and \cite[Theorem~5(ii)]{BNPS15}.
\end{proof}

In our next result, Theorem~\ref{thm:extra}, we study the KL property of $F(x,y):= f(x) + \frac\beta{2}\|x - y\|^2$ for any $\beta > 0$ when $f$ is a KL function. The function $F$ is used in the convergence analysis of various first-order methods whose iterates involve {\em momentum} terms; see, for example, \cite{BotCse15,OCBP14} for convergence analysis of some inertial proximal algorithms, and \cite{Chambolle2014,Johnstone2015} for convergence analysis of the proximal gradient algorithm with extrapolation. Theorem~\ref{thm:extra} will be used to analyze the convergence rate of an inertial proximal algorithm, the iPiano \cite{OCBP14}, with constant step-sizes, in Theorem \ref{thm:IP}.

We start with the following simple lemma.

\begin{lemma}\label{lem:prep3}
  Let $\gamma \in (1,2]$. Then there exist $\eta_1 > 0$ and $\eta_2\in (0,1)$ so that
  \[
  \|a + b\|^\gamma \ge \eta_1\|a\|^\gamma - \eta_2\|b\|^\gamma
  \]
  for any $a$, $b\in \R^n$.
\end{lemma}
\begin{proof}
  Notice that $x \mapsto \|x\|^\gamma$ is convex because $\gamma \in (1,2]$. Fix any $\lambda \in (0,\frac12)$. Then we have from convexity that
  \[
  \|a\|^\gamma = \left\|\lambda \frac{a + b}{\lambda} + (1-\lambda)\frac{-b}{(1-\lambda)}\right\|^\gamma \le \lambda^{1-\gamma}\|a + b\|^\gamma + (1-\lambda)^{1-\gamma}\|b\|^\gamma.
  \]
  Rearranging terms, we obtain further that
  \[
  \|a + b\|^\gamma \ge \lambda^{\gamma-1}\|a\|^\gamma - \left(\frac{\lambda}{1-\lambda}\right)^{\gamma-1}\|b\|^\gamma.
  \]
  The proof is completed upon noting that $\frac{\lambda}{1-\lambda} \in (0,1)$, since $\lambda \in (0,\frac12)$.
\end{proof}

\begin{theorem}{\bf (Exponent for a potential function for iPiano)}\label{thm:extra}
  Suppose that $f$ is a proper closed function that has the KL property at $\bar x\in {\rm dom}\,\partial f$ with an exponent of $\alpha\in [\frac12,1)$ and $\beta > 0$. Consider the function $F(x,y) := f(x) + \frac\beta{2}\|x - y\|^2$. Then the function $F$ has the KL property at $(\bar x,\bar x)$ with an exponent of $\alpha$.
\end{theorem}

\begin{proof}
  Since $f$ has the KL property at $\bar x$ with an exponent of $\alpha\in (0,1)$, there exist $c$, $\epsilon$ and $\nu>0$ so that
  \begin{equation}\label{ineq4}
    {\rm dist}^\frac1{\alpha}(0,\partial f(x)) \ge c (f(x) - f(\bar x))
  \end{equation}
  whenever $x\in {\rm dom}\,\partial f$, $\|x - \bar x\|\le \epsilon$ and $f(x) < f(\bar x)+\nu$, where the condition $f(\bar x)< f(x)$ is dropped because \eqref{ineq4} holds trivially otherwise.
  By shrinking $\epsilon$ further if necessary, we assume that $\epsilon < \frac1{2}$.

  Now, consider any $(x,y)$ satisfying $x\in {\rm dom}\,\partial f$, $\|x - \bar x\|\le \epsilon$, $\|y - \bar x\|\le \epsilon$ and $F(\bar x,\bar x) < F(x,y) < F(\bar x,\bar x) + \nu$. Clearly, any such $(x,y)$ satisfies
  \[
  f(x)\le F(x,y) < F(\bar x,\bar x) +\nu = f(\bar x)+\nu
  \]
  and thus \eqref{ineq4} holds for these $x$.
  Then for some suitable positive constants $C_0$, $C_1$ and $C_2$ (to be specified below), we have for any such $(x,y)$ that
  \begin{equation*}
    \begin{split}
      &{\rm dist}^\frac{1}{\alpha}(0,\partial F(x,y))  \ge C_0(\|\beta(y - x)\|^\frac{1}{\alpha} + \inf_{\xi \in \partial f(x)}\|\xi + \beta(x - y)\|^\frac{1}{\alpha})\\
      & \ge C_0(\|\beta(y - x)\|^\frac{1}{\alpha} + \inf_{\xi \in \partial f(x)}\eta_1\|\xi\|^\frac{1}{\alpha} - \eta_2\|\beta(x - y)\|^\frac{1}{\alpha})\\
      & \ge C_1(\inf_{\xi \in \partial f(x)}\|\xi\|^\frac{1}{\alpha} + \|\beta(y - x)\|^\frac{1}{\alpha}) \ge C_2\left(\inf_{\xi \in \partial f(x)}\|\xi\|^\frac{1}{\alpha} + \frac{\beta c}2\|y - x\|^\frac{1}{\alpha}\right)\\
      & \ge C_2c\left(f(x) - f(\bar x) + \frac\beta2\|y - x\|^\frac{1}{\alpha}\right)\\
      & \ge C_2c\left(f(x) - f(\bar x) + \frac{\beta}{2}\|y - x\|^2\right)  = C_2c(F(x,y) - F(\bar x,\bar x)),
    \end{split}
  \end{equation*}
  where the existence of $C_0>0$ in the first inequality follows from Lemma~\ref{ainequ}; the second inequality follows from Lemma~\ref{lem:prep3} applied to the term $\|\xi + \beta(x - y)\|^\frac{1}{\alpha}$, with $\eta_1$ and $\eta_2$ given by Lemma~\ref{lem:prep3}; the third inequality follows by setting $C_1 = C_0\min\{\eta_1,1-\eta_2\}$; the fourth inequality follows by further shrinking $C_1$, and $c$ is the same number as in \eqref{ineq4}; the fifth inequality follows from \eqref{ineq4}, while the last inequality follows from the assumption on $\alpha$ and the observation that 
  \[
  \|x - y\| \le \|x - \bar x\| + \|y - \bar x\|\le 2\epsilon < 1.
  \]
  This completes the proof.
\end{proof}

In our last theorem in this section, we examine KL property on a subset, more precisely, a manifold $\frak M$. Roughly speaking, we show that under partial smoothness and some additional assumptions, one only needs to verify the KL property along $\frak M$ in order to establish the property at a point $\bar x$. Before stating the result, we recall some necessary definitions.
First, following \cite[Definition~13.27]{Rock98}, a proper closed function $f$ is said to be prox-regular at a point $\ox$ for a subgradient $\ov\in \partial f(\ox)$ if $\ox\in {\rm dom}\,f$ and there
exists $\rho \ge 0$ such that
\[
f(x') \ge f(x) + \langle v, x' - x \rangle-\frac{\rho}{2}\|x' - x\|^2,
\]
whenever $x$ and $x'$ are near $\ox$ with $f(x)$ near $f(\ox)$ and $v \in \partial f(x)$ is near $\ov$. Furthermore, we say that $f$ is
prox-regular at $\ox$ if it is prox-regular at $\ox$ for every $\ov \in \partial f(\ox)$.

Next, let ${\frak M}$ be a ${\cal C}^2$ manifold about $\ox \in {\frak M}$.\footnote{Following \cite{HL04}, this notion means that locally $\frak M$ can be expressed as the solution set of a collection of ${\cal C}^2$ equations with linearly independent gradients.}
For a function $f$ that is ${\cal C}^2$ around $\ox\in\frak M$, the covariant derivative $\nabla_{\frak M}f(\ox)$ is defined as the unique vector in $T_\ox({\frak M})$ with
\[
  \langle\nabla_{\frak M}f(\ox),\xi\rangle = \left.\frac{d}{dt}f(P_{\frak M}(\ox + t\xi))\right|_{t=0}
\]
for any $\xi\in T_\ox({\frak M})$, where $T_\ox({\frak M})$ is the tangent space of $\frak M$ at $\ox$, and $P_{\frak M}$ is the projection onto $\frak M$; this latter operation is well defined in a sufficiently small neighborhood (in $\R^n$) of $\ox\in {\frak M}$.

Finally, we recall the notion of partial smoothness. From \cite[Definition 2.3]{HL04} (see also \cite[Definition~2.7]{Lewis02}), a function
$f$ is ${\cal C}^2$-partly smooth at $\ox$ relative to ${\frak M}$ if the
following four properties hold:
\begin{itemize}
\item[(i)] (restricted smoothness) the restriction of $f$ on ${\frak M}$  is a ${\cal C}^2$ function near $\ox$;
\item[(ii)] (regularity) at every point in ${\frak M}$ close to $\ox$, the function $f$ is regular and $\partial f(\ox)\neq \emptyset$;
\item[(iii)] (normal sharpness) the affine span of $\partial f(\ox)$ is a translate of the limiting normal cone $N_{{\frak M}}(\ox)$;
\item[(iv)] (subgradient continuity) the subdifferential mapping $\partial f$ restricted to ${\frak M}$ is continuous at $\ox$.
\end{itemize}
The class of prox-regular and ${\cal C}^2$-partly smooth function is a broad class of functions. For example, as pointed out in \cite[Example 2.5]{HL04},
any function $f$ that can be expressed as $f(x)=\max\{f_i(x): 1 \le i \le p\}$ for some $p \in \mathbb{N}$ and ${\cal C}^2$ functions $f_i$ is prox-regular. Moreover, in the same example, the authors also demonstrated that, if $\{\nabla f_i(\bar x)\}_{1 \le i \le p}$ is linearly independent, then $f$ is  ${\cal C}^2$-partly smooth at $\ox$ relative to ${\frak M}=\{x: I_f(x)=I_f(\bar x)\}$ where
$I_f(x)=\{i \in \{1,\ldots,p\}: f(x)=f_i(x)\}$.

\begin{theorem}{\bf (Exponent for partly smooth KL functions)}\label{thm:partlysmooth}
  Suppose that $f$ is a proper closed function that is prox-regular at $\bar x$ and ${\cal C}^2$-partly smooth at $\bar x$ relative to a manifold ${\frak M}$. Suppose further that $0\in {\rm ri}\,\partial f(\bar x)$ and that there exist $c$, $\nu$, $\epsilon > 0$ and $\alpha\in [0,1)$ so that
  \[
  \|\nabla_{\frak M}f(x)\|\ge c (f(x) - f(\bar x))^\alpha
  \]
  whenever $x\in \frak M$ with $\|x - \bar x\|\le \epsilon$ and $f(\bar x) < f(x) < f(\bar x) + \nu$.
  Then $f$ has the KL property at $\bar x$ with an exponent of $\alpha$.
\end{theorem}

\begin{proof}
  Our proof proceeds by contradiction. Suppose the contrary holds. Then there exists $\{x^k\}\subseteq {\rm dom}\,\partial f$ with $x^k \to \bar x$ and $f(x^k) > f(\bar x)$, $f(x^k)\to f(\bar x)$ such that
  \begin{equation}\label{eq:contra}
  {\rm dist}(0,\partial f(x^k)) < \frac1k (f(x^k) - f(\bar x))^\alpha.
  \end{equation}
  In particular, ${\rm dist}(0,\partial f(x^k))\to 0$. This together with the assumptions on prox-regularity, partial smoothness, $0\in {\rm ri}\,\partial f(\bar x)$ and \cite[Theorem~5.3]{HL04} shows that $x^k\in \frak M$ for all sufficiently large $k$. Thus, by considering even larger $k$ if necessary so that $\|x^k - \bar x\|\le \epsilon$ and $f(\bar x) < f(x^k) < f(\bar x) + \nu$, we conclude from the assumption that
  \[
  \|\nabla_{\frak M}f(x^k)\|\ge c (f(x^k) - f(\bar x))^\alpha
  \]
  for all sufficiently large $k$. Since ${\rm dist}(0,\partial f(x^k))= \|\nabla_{\frak M}f(x^k)\|$ for all sufficiently large $k$ as a consequence of \cite[Proposition~23]{DHM06}, we have obtained a contradiction to \eqref{eq:contra}. This completes the proof.
\end{proof}

\section{Structured problems: Luo-Tseng error bound and KL property}\label{sec:LT}

In this section, we examine structured optimization problems where the objective functions $f$ are proper closed functions taking the following form:
\begin{equation}\label{P1}
  f(x) := h(x) + P(x).
\end{equation}
Here, $h$ is a proper closed (possibly nonconvex) function with an open domain, and is continuously differentiable with a locally Lipschitz continuous gradient on ${\rm dom}\, h$, and $P$ is proper closed convex. Since $f$ is proper closed, we must have ${\rm dom}\,h\cap {\rm dom}\,P\neq \emptyset$.
For this class of functions, the Luo-Tseng error bound \eqref{eq:LTeb} is commonly used in the literature for establishing local linear convergence of various first-order methods applied to minimizing $f$; see, for example, \cite{LuoT92,LuoT92_2,LuoT93,Tse10,TseY09}. In this section, we study the relationship between the Luo-Tseng error bound and the KL property for the class of functions \eqref{P1}, under the following assumption concerning separation of stationary values. Recall that ${\cal X}$ is the set of stationary points of $f$.

\begin{assumption}\label{assum1}
  For any $\bar x\in \cal X$, there exists $\delta > 0$ so that $f(y) = f(\bar x)$ whenever $y\in {\cal X}$ and $\|y - \bar x\|\le \delta$.
\end{assumption}
This assumption is trivially satisfied if $h$ is, in addition, convex. Moreover, a global version of this assumption is commonly used together with the Luo-Tseng error bound for convergence analysis in the literature.

Before proving our main result of this section under Assumption~\ref{assum1}, we first establish the following auxiliary lemma. This lemma is a generalization of \cite[Proposition~1.5.14]{FchP03:vi}, where we have a general proper closed convex function $P$ instead of just the indicator function of a closed convex set.
\begin{lemma}\label{lem1}
  Consider the function $f$ given in \eqref{P1}.
  For any $x \in {\rm dom}\, \partial f$, it holds that
  \begin{equation}\label{eq:key}
  \|{\rm prox}_{P}(x - \nabla h(x)) - x\| \le {\rm dist}(0,\partial f(x)).
  \end{equation}
\end{lemma}
\begin{proof}
  For any $x \in {\rm dom}\, \partial f$, we have $\partial f(x)\neq \emptyset$. Since $\partial f(x) = \nabla h(x) + \partial P(x)$ thanks to \cite[Exercise~8.8(c)]{Rock98}, we must then have $\partial P(x)\neq \emptyset$. Consequently, the set $x + \partial P(x)$ is nonempty. Using the definition of proximal mapping, this means that there exists $z$ so that $x = {\rm prox}_{P}(z)$.

  Now, for any $x \in {\rm dom}\, \partial f$, let $z$ be such that $x = {\rm prox}_{P}(z)$. Then
  \begin{equation*}
    \begin{split}
      \|{\rm prox}_{P}(x - \nabla h(x)) - x\|&= \|{\rm prox}_{P}({\rm prox}_{P}(z) - \nabla h({\rm prox}_{P}(z))) - {\rm prox}_{P}(z)\|\\
      & \le \|{\rm prox}_{P}(z) - \nabla h({\rm prox}_{P}(z)) - z\|,
    \end{split}
  \end{equation*}
  where the inequality follows from \eqref{eq:prox_nonexpansive}. Consequently, we have
  \begin{equation}\label{eq:inf1}
  \|{\rm prox}_{P}(x - \nabla h(x)) - x\|
  \le \inf_{z}\{\|{\rm prox}_{P}(z) - \nabla h({\rm prox}_{P}(z)) - z\|:\; x = {\rm prox}_{P}(z)\}.
  \end{equation}
  On the other hand, observe that we have
  \begin{equation}\label{eq1}
    \begin{split}
      &{\rm dist}(0,\partial f(x))  = \inf_\xi\{ \|\nabla h(x) + \xi\| :\; \xi \in \partial P(x)\}
      = \inf_\xi\{\|\nabla h(x) - x + \underbrace{x + \xi}_{z}\| :\; \xi \in \partial P(x)\}\\
      & = \inf_z \{ \|\nabla h(x) - x + z\| :\; z \in x + \partial P(x)\} = \inf_z \{ \|\nabla h(x) - x + z\| :\; x = {\rm prox}_{P}(z)\}\\
      & = \inf_{z} \{ \|\nabla h({\rm prox}_{P}(z)) - {\rm prox}_{P}(z) + z\|:\; x = {\rm prox}_{P}(z)\},
    \end{split}
  \end{equation}
  where the fourth equality follows from the definition of the proximal mapping. The desired inequality \eqref{eq:key} now follows by combining \eqref{eq:inf1} and \eqref{eq1}.
\end{proof}

Using the above lemma, we can now prove the following result, which states that if the Luo-Tseng error bound and Assumption~\ref{assum1} hold, then $f$ is a KL function with an exponent of $\frac12$.
\begin{theorem}{\bf (Luo-Tseng error bound implies KL)}\label{LT_to_KL}
  Suppose that ${\cal X}\neq \emptyset$, and that Assumption~\ref{assum1} and the Luo-Tseng error bound hold. Then $f$ in \eqref{P1} is a KL function with an exponent of $\frac12$.
\end{theorem}
\begin{proof}
From Lemma \ref{lem:prep2}, we only need to show that $f$ has the KL property at any $\bar x\in \cal X$ with an exponent of $\frac12$. To see this, fix any $\bar x\in {\cal X}$.
Let $\delta > 0$ be defined as in Assumption~\ref{assum1} and take any $\epsilon \in(0,\frac{\delta}{4})$ such that $B(\bar x,2\epsilon)\subset {\rm dom}\,h$ (this is possible as ${\rm dom}\,h$ is open and $\bar x \in {\rm dom}\, h$). Then for any $x$ with $\|x - \bar x\|\le \epsilon$,
  we have
  \[
  \|u - \bar x\| \le \|u - x\| + \|x - \bar x\|= {\rm dist}(x,{\cal X})+ \|x - \bar x\| \le 2\|x - \bar x\| \le 2\epsilon < \frac\delta{2} < \delta
  \]
  whenever $u \in {\rm Proj}_{\bar{\cal X}}(x)$. From this, one can argue that ${\rm Proj}_{\bar{\cal X}}(x) = {\rm Proj}_{\cal X}(x)$ for these $x$. Indeed, if $u \in {\rm Proj}_{\bar{\cal X}}(x)\subseteq \bar{\cal X}$, then there exist $u^k \to u$ satisfying $0 \in \nabla h(u^k) + \partial P(u^k)$ for all $k$. Using this, $u \in B(\bar x,2\epsilon)\subset {\rm dom}\,h$ and the closedness of $\partial P$, we conclude that $u \in \cal X$. This proves ${\rm Proj}_{\bar{\cal X}}(x) = {\rm Proj}_{\cal X}(x)$ since ${\rm Proj}_{\bar{\cal X}}(x) \supseteq {\rm Proj}_{\cal X}(x)$ holds trivially. Thus, it holds that $f(u) = f(\bar x)$ for any $u \in {\rm Proj}_{\cal X}(x)$ whenever $\|x - \bar x\|\le \epsilon$.
On the other hand, notice that $B(\bar x,2\epsilon)$ is a compact subset of ${\rm dom}\,h$ and $\nabla h$ is locally Lipschitz on ${\rm dom}\,h$. Consequently, $\nabla h$ is {\em globally} Lipschitz continuous on $B(\bar x,2\epsilon)$; we denote its Lipschitz constant by $L > 0$.

Next, from the assumption on Luo-Tseng error bound, we see that for $\zeta = f(\bar x)+1$, there exist $c_1$, $\epsilon_1 > 0$ so that
  \begin{equation}\label{copyerrorbd}
  {\rm dist}(x,{\cal X}) \le c_1\|{\rm prox}_P(x - \nabla h(x)) - x\|
  \end{equation}
  whenever $\|{\rm prox}_P(x - \nabla h(x)) - x\| < \epsilon_1$ and $f(x)\le f(\bar x) + 1$. Since $\|{\rm prox}_P(x - \nabla h(x)) - x\|< \epsilon_1$ is a neighborhood of ${\cal X}$ containing $\bar x$, by shrinking $\epsilon$ if necessary, we may assume that \eqref{copyerrorbd} holds whenever $\|x - \bar x\|\le \epsilon$ and $f(x)\le f(\bar x) + 1$.

  From the above discussions, for any $x\in {\rm dom}\,\partial f$ with $f(\bar x) < f(x) < f(\bar x) + 1$ and $\|x-\bar x\| \le \epsilon$, we have $f(\bar x) = f(u)$ for any $u\in {\rm Proj}_{\cal X}(x)$, and
  \[
  \begin{split}
    &f(x) - f(\bar x) = f(x) - f(u) = h(x) - h(u) + P(x) - P(u)\\
    & \le \langle\nabla h(u),x - u\rangle + \frac{L}{2}\|x - u\|^2 + \langle \xi, x - u\rangle\\
    & = \langle\nabla h(u) - \nabla h(x),x - u\rangle + \frac{L}{2}\|x - u\|^2 + \langle  \nabla h(x) + \xi, x - u\rangle\\
    & \le \frac{3L}2\|x - u\|^2 + \|\nabla h(x) + \xi\|\cdot\|x - u\|\\
    & = \frac{3L}2{\rm dist}^2(x,{\cal X}) + \|\nabla h(x) + \xi\|\cdot{\rm dist}(x,{\cal X})\\
    & \le C_0(\|{\rm prox}_P(x - \nabla h(x)) - x\|^2 + \|\nabla h(x) + \xi\|\cdot \|{\rm prox}_P(x - \nabla h(x)) - x\|)
  \end{split}
  \]
  for any $\xi \in \partial P(x)$, where the first inequality is a consequence of the subgradient inequality applied to $P$ and the fact that $\nabla h$ is Lipschitz continuous with a Lipschitz constant of $L$ on $B(\bar x,2\epsilon)$, which contains both $x$ and $u$, the last equality follows from the definition of $u$, while the last inequality follows from \eqref{copyerrorbd} for some suitable $C_0 > 0$. Taking infimum over all possible $\xi \in \partial P(x)$ and invoking Lemma~\ref{lem1}, we see further that
  \[
  f(x) - f(\bar x)  \le C_1[{\rm dist}(0,\partial f(x))]^2
  \]
  for some $C_1 > 0$. This completes the proof.
\end{proof}

Before ending this section, we discuss some immediate applications of Theorem~\ref{LT_to_KL}. The first application involves piecewise linear-quadratic (PLQ) functions\footnote{Recall that a proper closed function $F$ is called piecewise linear-quadratic \cite[Definition~10.20]{Rock98} if ${\rm dom}\,F$ can be represented as the union of finitely many polyhedrons, relative to each of which $F(x)$ is given by the form $\frac{1}{2}x^TMx+a^Tx+\alpha$, where $M \in {\cal S}^n$, $a \in \mathbb{R}^n$ and $\alpha \in \mathbb{R}$.} 
\begin{equation}\label{eq:structure0}
f(x)=l(A x)+P (x),
\end{equation}
where $l$ is strongly convex on any compact convex set and is twice continuously differentiable, $A\in \R^{m\times n}$ and $P$ is a convex PLQ function on $\mathbb{R}^n$. We show in the next proposition that the $f$ in \eqref{eq:structure0} is a KL function with an exponent of $\frac12$.
  We would also like to point out that, in the special case of $\ell_1$-regularized least squares function (that is,
   $f(x)=l(Ax)+P(x)$ where $l(y)=\frac12\|y-b\|^2$ and $P(x)=\mu\sum_{i=1}^n|x_i|$ for some $\mu > 0$), it was known in \cite[Section 3.2.1]{BNPS15} that $f$ is a KL function with an exponent of $\frac12$.

  To achieve our promised result, we will need the notions of metric subregularity and calmness for set-valued mappings.
  Recall that a set-valued mapping $G:\mathbb{R}^n \rightarrow 2^{\mathbb{R}^m}$ is said to be {\rm metrically subregular} \cite[Page~183]{Donchev} at $\ox$ for $\oy$ if $\oy\in G(\ox)$, and there exist a constant $\kappa>0$, a neighborhood $U$ of $\ox$ and a neighborhood $V$ of $\oy$ such that
\begin{equation*}
{\rm dist}(x,G^{-1}(\oy))\le \kappa {\rm dist}(\oy, G(x) \cap V)\quad \mbox{for all}\quad x\in U,
\end{equation*}
A property closely related to metric subregularity is the notion of calmness (see \cite[Page~182]{Donchev}). A mapping $S : \R^m \rightarrow
2^{\R^n}$ is said to be calm at $\oy$ for $\ox$ if $\ox \in S(\oy)$, and
there is a constant $r> 0$ along with neighborhoods $U$ of $\ox$ and $V$ of $\oy$ such that
\[
S(y)\cap U \subseteq S(\oy)+r \|y-\oy\| \mathbb{B}_{\mathbb{R}^n} \mbox{ for all }y \in V,
\]
where $\mathbb{B}_{\mathbb{R}^n}$ is the closed unit ball in $\R^n$.
It is known that a set-valued mapping is calm if and only if its inverse mapping is metrically sub-regular \cite[Theorem 3H.3]{Donchev}.

\begin{proposition}\label{thm:cplq} {\bf (Convex problems with convex piecewise linear-quadratic regularizers)}
Consider a convex function $f$ given as in \eqref{eq:structure0}. Suppose that $\Argmin f$ is a nonempty compact set. Then $f$ is a KL function with an exponent of $\frac{1}{2}$.
\end{proposition}
\begin{proof}
Fix any $\bar x\in \Argmin f$. Denote $\bar y=A\bar x$ and $\bar g=A^*\nabla l(A\bar x)$. Then $-\bar g \in \partial P(\bar x)$ from the first-order optimality condition and \cite[Exercise~8.8(c)]{Rock98}. Notice that $Ax\equiv \bar y$ over $\Argmin f$ due to the strong convexity of $l$ on compact convex sets. Consequently, we also have $A^*\nabla l(Ax) \equiv A^*\nabla l(\bar y)= \bar g$ over $\Argmin f$. Define
\[
\Gamma_f(\bar y) := A^{-1}\bar y\ \ {\rm and}\ \ \Gamma_P(\bar g) := \partial P^*(-\bar g)
\]
as in \cite[Section~3.3]{ZSo15}. We will subsequently show that
\begin{enumerate}[{\rm (1)}]
  \item The pair of sets $\{\Gamma_f(\bar y),\Gamma_P(\bar g)\}$ is boundedly linearly regular.
  \item The subdifferential mapping $\partial P$ is metrically sub-regular at $\bar x$ for $-\bar g$.
\end{enumerate}
Granting these, since $\bar x \in \Argmin f$ is arbitrary, we conclude using \cite[Theorem~2]{ZSo15} and \cite[Corollary~1]{ZSo15} that the Luo-Tseng error bound holds. Since $f$ is convex so that Assumption~\ref{assum1} is trivially satisfied, the desired conclusion follows from Theorem~\ref{LT_to_KL}.

Now it remains to prove the two claims above.

For (1), note that $P^*$ is convex piecewise linear-quadratic according to \cite[Theorem 11.14]{Rock98}, which implies that $\partial P^*(-\bar g)$ is a polyhedral set \cite[Proposition 10.21]{Rock98}. Consequently, it follows from \cite[Corollary 5.2.6]{BB96} that the pair of sets $\{\Gamma_f(\bar y),\Gamma_P(\bar g)\}$ is linearly regular in the sense that there exists $c>0$ such that \[{\rm dist}\big(x,\Gamma_f(\bar y)\cap \Gamma_P(\bar g)\big) \le c \, \max \{{\rm dist}\big(x,\Gamma_f(\bar y)\big),{\rm dist}\big(x,\Gamma_P(\bar g)\big)\} \mbox{ for all } x \in \mathbb{R}^n. \]
Hence, this pair of set is in particular boundedly linearly regular; see \cite[Definition~5.6]{BB96}.

For (2), recall that $\partial P^* = (\partial P)^{-1}$ is a piecewise polyhedral set-valued mapping \cite[Proposition 12.30]{Rock98}. Thus, the set-valued mapping $(\partial P)^{-1}$ is calm by Robinson's theorem on calmness of piecewise affine mappings \cite{Robinson} (see also \cite[Example 9.57]{Rock98}). Using this and the fact that a set-valued mapping is calm if and only if its inverse mapping is metrically sub-regular \cite[Theorem 3H.3]{Donchev}, we conclude further that $\partial P$ is metrically sub-regular at $\bar x$ for $-\bar g\in \partial P(\bar x)$. This completes the proof.
\end{proof}

In our second application, we consider the following model
\begin{equation}\label{eq:structure1}
  f(x) = l(Ax) + \delta_C(x),
\end{equation}
where $l$ is strongly convex on any compact convex set and is continuously differentiable with Lipschitz gradient, and $A\in \R^{m\times n}$. In addition, the set $C$ is defined as
\[
C:= \left\{x:\; \sum_{i=1}^mw_i\|x_i\|_p \le \sigma\right\},
\]
where $\sigma > 0$, $p \in [1,2]$, and for each $i=1,\ldots,m$, $x_i\in \R^{n_i}$ with $\sum_{i=1}^{m}n_i = n$, $w_i > 0$, and
\[
\|x_i\|_p := \left(\sum_{j=1}^{n_i}|(x_i)_j|^p\right)^{\frac1p}.
\]
When $p = 2$ and $l(y) = \frac12\|y - b\|^2$ for some $b$, the model \eqref{eq:structure1} can be viewed as a variant of the group LASSO problem \cite{YL06}. In the next proposition, we show that the function in \eqref{eq:structure1} has the KL property with an exponent of $\frac12$, under mild assumptions.

\begin{proposition}\label{prop:group_lasso}
  Consider a convex function $f$ given as in \eqref{eq:structure1}. Suppose that $\inf_{x\in \R^n}f(x) > \inf_{x\in \R^n}l(Ax)$. Then $f$ is a KL function with an exponent of $\frac{1}{2}$.
\end{proposition}
\begin{proof}
  Let $g(x) := \sum_{i=1}^mw_i\|x_i\|_p - \sigma$. Then $g(0) = -\sigma< 0$. Now, fix any $\lambda > 0$ and consider $h_\lambda(x) = l(Ax) + \lambda g(x)$. We first consider the case $\Argmin(h_\lambda)\neq \emptyset$. In this case, in view of \cite[Proposition~1]{ZZhSo15}, the set $\Argmin(h_\lambda)$ is compact. Thus, if $p=1$, we conclude from \cite[Corollary~1]{ZZhSo15} that the Luo-Tseng error bound holds for $h_\lambda$. On the other hand, if $p\in (1,2]$, then it follows from \cite[Corollary~2]{ZZhSo15} that the Luo-Tseng error bound also holds for $h_\lambda$. Since $h_\lambda$ is convex so that Assumption~\ref{assum1} is trivially satisfied, we see further from Theorem~\ref{LT_to_KL} that $h_\lambda$ is a KL function with an exponent of $\frac12$ in this case. Note that the same conclusion holds when $\Argmin(h_\lambda)=\emptyset$ due to Lemma~\ref{lem:prep2}. Consequently, the assumptions in Theorem~\ref{thm:Lagrange} are satisfied with $D = \R^n$ and $h = l\circ A$, and the desired conclusion follows immediately from this theorem.
\end{proof}

\section{Applications}\label{sec6}

In this section, we apply our results in the previous sections to deducing the KL exponent of some specific functions that arise in various applications. We then discuss how our results can be applied to establishing the linear convergence of some first-order methods.

\subsection{Piecewise linear regularizers}\label{sec5.1}

In this subsection, we demonstrate how our results can be applied to some problems with possibly nonconvex piecewise linear regularizers, i.e., their epigraphs are unions of polyhedrons. In particular, we have the following corollary.

\begin{corollary}\label{cor6}
  Suppose that $f$ is a proper closed function taking the form
  \begin{equation*}
    f(x) = l(Ax) + \min_{1\le i\le r}P_i(x),
  \end{equation*}
  where $A\in \mathbb{R}^{m\times n}$, $P_i$ are proper closed polyhedral functions for $i=1,\ldots,r$ and $l$ satisfies either one of the following two conditions:
  \begin{enumerate}[{\rm (i)}]
      \item $l$ is a proper closed convex function with an open domain, and is strongly convex on any compact convex subset of ${\rm dom}\,l$ and is twice continuously differentiable on ${\rm dom}\,l$.

    \item $l(y) = \max_{u\in D}\{\langle y,u\rangle - q(u)\}$ for all $y\in \R^m$, with $D$ being a polyhedron and $q$ being a strongly convex differentiable function with a Lipschitz continuous gradient.
  \end{enumerate}
  Suppose in addition that $f$ is continuous on ${\rm dom}\,\partial f$. Then $f$ is a KL function with an exponent of $\frac12$.
\end{corollary}
\begin{proof}
  Since $f$ is proper, we can assume without loss of generality that ${\rm dom}\,l\cap A{\rm dom}\,P_i\neq \emptyset$ for $i=1,\ldots,r$. Let $f_i(x) = l(Ax) + P_i(x)$ for $i=1,\ldots,r$. Consider those $f_i$ with $\Argmin f_i\neq \emptyset$. Suppose first that condition (i) holds. Our arguments follow the proof of \cite[Lemma~7]{TseY09}. Define
  \[
  g_i(x,s) = \underbrace{l(Ax) + s}_{h(x,s)} + \delta_{K_i}(x,s) = l\left(\begin{bmatrix}
    A&0
  \end{bmatrix}\begin{bmatrix}
    x\\s
  \end{bmatrix}\right) + \begin{bmatrix}
    0\\1
  \end{bmatrix}^T\begin{bmatrix}
    x\\s
  \end{bmatrix} + \delta_{K_i}(x,s).
  \]
  where $K_i=\{(x,s): P_i(x)\le s\}$. In particular, $h$ takes the form of \cite[Eq~1.1]{LuoT92}.
  Also, one can check that $l$ satisfies the assumptions 1.1 and 1.2 in \cite{LuoT92}.\footnote{Assumption 1.1(a) in \cite{LuoT92} holds because ${\rm dom}\,l$ is open and $l$ is proper. Assumption 1.1(b) and Assumption 1.2(b) in \cite{LuoT92} hold as $l$ is strongly convex on any compact convex subset of ${\rm dom}\,l$ and is twice continuously differentiable on ${\rm dom}\,l$. Assumption 1.2(a) in \cite{LuoT92} holds because we are considering the case that $\Argmin f_i\neq \emptyset$ and so $\Argmin g_i\neq \emptyset$. Finally, assumption 1.1(c) in \cite{LuoT92} holds because $l$ is lower semicontinuous with an open domain, so that for any $\bar y$ in the boundary of the domain, one has $\liminf\limits_{y\to\bar y}l(y) \ge l(\bar y) = \infty$.} Thus, in view of \cite[Theorem~2.1]{LuoT92}, we conclude that the Luo-Tseng error bound holds for $g_i$, i.e., for any $\zeta\ge \inf g_i = \inf f_i$, there exist $C$ and $\epsilon> 0$ so that
  \begin{equation}\label{LT_for_gi}
  {\rm dist}((x,s),\Argmin g_i) \le C\|{\rm Proj}_{K_i}[(x,s) - \nabla h(x,s)] - (x,s)\|
  \end{equation}
  whenever $\| {\rm Proj}_{K_i}[(x,s) - \nabla h(x,s)] - (x,s)\| < \epsilon$ and $g_i(x,s)\le \zeta$.
  Moreover, note that we have
  \[
  \|{\rm Proj}_{K_i}[(x,P_i(x)) - \nabla h(x,P_i(x))] - (x,P_i(x))\| \le \kappa \|{\rm prox}_{P_i}(x - \nabla (l\circ A)(x)) - x\|
  \]
  whenever $x\in {\rm dom}\,f_i$ for some $\kappa > 0$, thanks to \cite[Lemma~6]{TseY09},\footnote{The statement of \cite[Lemma~6]{TseY09} is proved under the assumption that $x\mapsto l(Ax)$ is smooth on an open set containing ${\rm dom}\,P_i$, but it is not hard to see that the proof is valid also in our settings, i.e., when ${\rm dom}\,l\cap A{\rm dom}\,P_i\neq \emptyset$ and ${\rm dom}\,l$ is open. For the convenience of the readers, we include a proof in the appendix.} and that for these $x$, it is easy to show that
  \[
  {\rm dist}(x,\Argmin f_i) \le {\rm dist}((x,P_i(x)),\Argmin g_i).
  \]
  Combining these two observations with \eqref{LT_for_gi}, we conclude that the Luo-Tseng error bound holds for $f_i$ whenever $\Argmin f_i\neq \emptyset$, i.e., for any $\zeta\ge \inf f_i$, there exist $C_1$ and $\epsilon_1> 0$ so that
  \[
  {\rm dist}(x,\Argmin f_i) \le C_1\|{\rm prox}_{P_i}(x - \nabla (l\circ A)(x)) - x\|
  \]
  whenever $\|{\rm prox}_{P_i}(x - \nabla (l\circ A)(x)) - x\| < \epsilon_1$ and $f_i(x)\le \zeta$.
  On the other hand, Suppose that condition (ii) holds. Then it follows from \cite[Theorem~4]{TseY09} that the Luo-Tseng error bound also holds for $f_i$. Next, observe that Assumption~\ref{assum1} is trivially satisfied for all $f_i$ because $f_i$ is convex for all $i$. Using these, Theorem~\ref{LT_to_KL} and Lemma~\ref{lem:prep2} (for the case $\Argmin f_i = \emptyset$), we conclude that the functions $f_i(x) = l(Ax) + P_i(x)$, $1\le i\le r$ are all KL functions with exponents $\frac12$.

  Finally, notice that $l$ has an open domain and is continuously differentiable on it under either condition (i) or (ii); indeed, under condition (ii), $l$ is the conjugate of the strongly convex function $u\mapsto q(u)+\delta_D(u)$, and is thus continuously differentiable everywhere thanks to \cite[Theorem~18.15]{BauCom10}. Consequently, in view of \cite[Exercise~8.8(c)]{Rock98}, we have $\partial f_i(x) = \nabla (l\circ A)(x) + \partial P_i(x)$ for any $x$. From this, we deduce further that, for each $i$,
  \[
  {\rm dom}\,f_i = {\rm dom}(l\circ A)\cap {\rm dom}\,P_i = {\rm dom}(l\circ A)\cap {\rm dom}\,\partial P_i = {\rm dom}\nabla(l\circ A)\cap {\rm dom}\,\partial P_i = {\rm dom}\,\partial f_i,
  \]
  where the second equality is a consequence of the polyhedricity of $P_i$: polyhedral functions are piecewise affine on their domains; see \cite[Proposition~5.1.1]{BorLewis06}. The desired conclusion now follows from Corollary~\ref{cor:minKL}.
\end{proof}

With this corollary in mind, we can show that the following classes of proper closed functions have KL exponents of $\frac12$:
\begin{equation}\label{example1}
\begin{split}
  f_1(x) &= l(Ax) + \mu \|x\|_1 - \mu\gamma \sum_{i=1}^k|x_{[i]}|,\\
  f_2(x) &= l(Ax) + \delta_{\|\cdot\|_0\le r}(x),\\
  f_3(x) &= l(Ax) + \delta_{\|\cdot\|_0\le r}(x) + \delta_{\Delta}(x),
\end{split}
\end{equation}
where: $A\in \mathbb{R}^{m\times n}$ is a measurement matrix and $l$ is defined as in Corollary~\ref{cor6}, $|x_{[i]}|$ is the $i$-th largest (in magnitude) entry in $x$, $k\le n$, $\mu > 0$, $\gamma \in (0,1]$, $r > 0$ and $\Delta:=\{x\in \R^n:\; e^Tx = 1,\ x\ge 0\}$. Indeed, it is clear that the regularizers on the right hand side in \eqref{example1} are representable as the minimum of finitely many proper closed polyhedral functions. Specifically,
\begin{equation*}
\begin{split}
 \mu \|x\|_1 - \mu\gamma \sum_{i=1}^k|x_{[i]}| &= \min_{I\in {\cal I}_k} \left\{\mu \|x\|_1 - \mu\gamma \sum_{i\in I}|x_{i}|\right\},\\
\delta_{\|\cdot\|_0\le r}(x) &= \min_{I \in {\cal I}_{n-r}} \left\{\delta_{H_I}(x)\right\},\\
\delta_{\|\cdot\|_0\le r}(x) + \delta_{\Delta}(x) &= \min_{I \in {\cal I}_{n-r}} \left\{\delta_{H_I\cap \Delta}(x)\right\},
\end{split}
\end{equation*}
where ${\cal I}_k$ is the collection of all subsets of $\{1,\ldots,n\}$ of size $k$ and $H_I = \{x:\; x_i = 0\ \ \forall i\in I\}$.
Thus, the conclusions on the KL exponents follow immediately from Corollary~\ref{cor6}.

The functions in \eqref{example1} arise as models for sparse recovery. Specifically, the function $f_1$ with $l(y) = \frac12\|y - b\|^2$ was considered in \cite{WLZ15}, the function $f_2$ with $l(y) = \frac12 \|y - b\|^2$ was considered in \cite{BluDav08} and many subsequent work, while the function $f_3$ with $l(y) = \frac12 \|y - b\|^2$ arises in the problem of sparse index tracking in portfolio optimization; see, for example, \cite[Section~8]{KBCK13}. In addition, we would like to point out that our assumptions on $l$ are general enough to cover the loss functions for logistic and Poisson regression, where $l(y) = \sum_{i=1}^m\log(1 + \exp(b_iy_i))$ and $l(y) = \sum_{i=1}^m(-b_iy_i + \exp(y_i))$ respectively, for some suitable $b\in \R^m$; see \cite[Table~1]{ZSo15}.
Moreover, one can also consider  functions of the form $l(y) = -\sum_{i=1}^m\log(y_i)$ that are not defined on the whole space. This makes it possible to analyze the local convergence rate for first-order methods applied to some models that involve these functions.

\subsection{Some nonconvex minimum-of-quadratic regularizers}\label{sec5.2}

In this subsection, we demonstrate how our results can be applied to some least squares problems with regularizers that can be written as the minimum of finitely many functions, each of which is the sum of a quadratic function (not necessarily convex) and a proper closed polyhedral function.

Specifically, we consider the following class of functions:
\begin{equation}\label{piecewisef}
f(x)=\min_{1 \le i \le r}\left\{\frac{1}{2}x^TM_ix+u_i^Tx+\beta_i+ P_i(x)\right\},
\end{equation}
where $P_i$ are proper closed polyhedral functions, $M_i\in \cS^n$, $u_i\in \mathbb{R}^n$ and $\beta_i \in \mathbb{R}$, for $i=1,\ldots,r$.  This class of functions covers nonconvex functions which need not be piecewise linear-quadratic,\footnote{For a simple example, consider $f(x)=-|x_1^2+x_2^2-1|$. Clearly, $f$ can be written as a form of \eqref{piecewisef}; while $f$ is not piecewise linear-quadratic because the pieces of this function cannot be arranged as a polyhedral union.} and encompasses a lot of problems that arise in applications, as we shall see later. We now state the following result.
\begin{corollary}\label{thm2}
Let $f$ be defined as in \eqref{piecewisef}. Suppose in addition that $f$ is continuous on ${\rm dom}\,\partial f$. Then $f$ is a KL function with an exponent of $\frac12$.
\end{corollary}
\begin{proof}
  Let $f_i(x) := \frac{1}{2}x^TM_ix+u_i^Tx+\beta_i+ P_i(x)$ for each $i=1,\ldots,r$. When the set of stationary points of $f_i$ is nonempty, it follows from \cite[Theorem~4]{TseY09} that the Luo-Tseng error bound holds. In addition, Assumption~\ref{assum1} can be seen to hold by applying \cite[Lemma~3.1]{LuoT92_2} to the problem $\min\{\frac{1}{2}x^TM_ix+u_i^Tx + \zeta:\; P_i(x)\le \zeta\}$. Thus, for these $f_i$, we have from Theorem~\ref{LT_to_KL} that they are KL functions with an exponent of $\frac12$. On the other hand, if $f_i$ does not have stationary points, then we see from Lemma~\ref{lem:prep2} that $f_i$ is a KL function with an exponent of $\frac12$. Finally, note from \cite[Exercise~8.8(c)]{Rock98} that ${\rm dom}\,\partial f_i = {\rm dom}\,\partial P_i$. Thus, it holds that, for each $i$,
  \[
  {\rm dom}\,\partial f_i = {\rm dom}\,\partial P_i = {\rm dom}\,P_i = {\rm dom}\,f_i,
  \]
  where the second equality follows from the polyhedricity of $P_i$, as these functions are piecewise affine in their domains; see \cite[Proposition~5.1.1]{BorLewis06}.
  The conclusion of this corollary now follows from Corollary~\ref{cor:minKL}.
\end{proof}

 We would like to point out that in the special case where $r=1$ and $P_1(x)=\delta_{C}(x)$ with $C=\{x:Ax \le b\}$, it was established in \cite[Theorem~1]{FNQ06} that the KL inequality holds with an exponent of $\frac12$  under the {\em additional assumption} that $C$ is compact and contains the origin in its interior. As we will see below, our extension here allows us to cover many popular optimization problems with nonconvex regularizers (such as the SCAD and MCP regularizers); while
\cite[Theorem~1]{FNQ06} cannot be applied due to the additional compactness assumption of the domain.

In details, we consider the least squares problems with SCAD \cite{Fan97} or MCP \cite{Zhang10} regularization functions. Recall that the objective function of these problems takes the form
\begin{equation}\label{equation3}
f(x) = \frac12 \|Ax - b\|^2 + \sum_{i=1}^nr(x_i),
\end{equation}
where $A\in \R^{m\times n}$ is a measurement matrix and $r:\R \to \R$ is a continuous function. For SCAD, the continuous function $r$ takes the form
\[
r(t) = \begin{cases}
  \lambda|t| & {\rm if}\ |t|\le \lambda,\\
  \frac{-t^2 + 2\theta\lambda|t| - \lambda^2}{2(\theta-1)} & {\rm if}\, \lambda < |t|\le \theta\lambda,\\
  \frac{(\theta+1)\lambda^2}{2} & {\rm if}\, |t|> \theta\lambda,
\end{cases}
\]
where $\lambda > 0$ and $\theta > 2$;
while for MCP, the function $r$ is given by
\[
r(t) = \begin{cases}
  \lambda|t| - \frac{t^2}{2\theta} & {\rm if}\, |t| \le \theta\lambda,\\
  \frac{\theta\lambda^2}{2} & {\rm if}\, |t| > \theta\lambda,
\end{cases}
\]
where $\lambda > 0$ and $\theta > 0$. In view of the definitions of SCAD and MCP, one can see that the regularization functions $\sum_{i=1}^n r(x_i)$ are continuous functions taking the form
\[
\sum_{i=1}^n \min_{1\le l\le m_i}\{f_{i,l}(x_i) + \delta_{C_{i,l}}(x_i)\}
\]
for some closed intervals $C_{i,l}$ and one dimensional quadratic (or linear) functions $f_{i,l}$, for each $1\le l\le m_i$ and $i=1,\ldots,n$. Since this is a sum of functions with independent variables, it can be equivalently written as
\[
\min_{j \in {\frak J}}\left\{\sum_{i=1}^n [f_{i,j_i}(x_i) + \delta_{C_{i,j_i}}(x_i)]\right\},
\]
where ${\frak J} = \{(j_1,\ldots,j_n)\in \mathbb{N}^n:\; 1\le j_i\le m_i\ \forall i\}$ and $\mathbb{N}$ is the set of positive integers.
Hence, the objective in the corresponding regularized least squares problems can be written as
\[
\min_{j \in {\frak J}}\left\{\frac12\|Ax - b\|^2 + \sum_{i=1}^n [f_{i,j_i}(x_i) + \delta_{C_{i,j_i}}(x_i)]\right\}
\]
to which Corollary~\ref{thm2} is applicable. Thus, the function $f$ in \eqref{equation3} with SCAD or MCP regularizers has a KL exponent of $\frac12$.

\subsection{Local linear convergence of some first-order methods}

In this subsection, we discuss local linear convergence of two common first-order methods, the proximal gradient algorithm and the inertial proximal algorithm, based on the KL exponent.

We first discuss the proximal gradient algorithm, also known as the forward-backward splitting algorithm. This algorithm has been studied extensively in recent years; see, for example, \cite{AtBoSv13} and references therein. The algorithm, in its general form, is applicable to the following optimization problem
\[
\min_{x \in \mathbb{R}^n} f(x) := h(x)+g(x),
\]
where $h$ is a smooth function whose gradient is Lipschitz continuous with constant $L$ and $g$ is a proper closed function. The proximal gradient algorithm for this problem can be formulated as
\begin{eqnarray}\label{FBSupdate}
  x^{k+1} \in \Argmin_{x\in \R^n}\left\{\langle \nabla h(x^k),x - x^k\rangle + \frac{1}{2\gamma_k}\|x - x^k\|^2 + g(x) \right\}.
\end{eqnarray}
where $0< \inf_{k\ge 0}\gamma_k \le \sup_{k\ge 0}\gamma_k < \frac{1}{L}$. The global convergence of this algorithm has been discussed in \cite[Section~5.2]{AtBoSv13}. We state below the local linear convergence of the method under an explicit assumption on the KL exponent, which is an immediate consequence of \cite[Theorem~3.4]{FrGaPe15}. This result will be used together with results in Sections~\ref{sec5.1} and \ref{sec5.2} for deriving new local convergence results when the proximal gradient algorithm is applied to some concrete optimization problems; see Remark~\ref{rem5.1} below.

\begin{proposition}\label{thm:FBS}
Suppose that $\inf f > -\infty$ and that $f$ is a KL function with an exponent of $\frac12$. Let $\{x^k\}$ be the sequence generated by the proximal gradient method given in \eqref{FBSupdate}. Suppose that $\{x^k\}$ is bounded. Then $\{x^k\}$ converges locally linearly to a stationary point of $f$.
\end{proposition}

The next algorithm we consider is the inertial proximal algorithm, known as iPiano, proposed and studied in \cite{OCBP14}. The algorithm can be applied to minimizing \eqref{P1} when ${\rm dom}\,h=\R^n$ and $\nabla h$ is globally Lipschitz continuous with Lipschitz constant $L$, and its algorithmic framework can be described below: set $x^{-1} = x^0$ and update
\begin{eqnarray}\label{iPianoupdate}
  x^{k+1} = \argmin_{x\in \R^n}\left\{\left\langle \nabla h(x^k) - \frac{\beta_k}{\alpha_k}(x^k-x^{k-1}),x - x^k\right\rangle + \frac{1}{2\alpha_k}\|x - x^k\|^2 + P(x) \right\}.
\end{eqnarray}
where $\{\beta_k\}$ and $\{\alpha_k\}$ are suitable choices of sequence specified in \cite[Algorithm~5]{OCBP14} and \cite[Theorem~4.9]{OCBP14}. In particular, the global convergence of the iPiano has been established in \cite[Theorem~4.9]{OCBP14} by assuming that $\{\beta_k\}\subset [0,1)$, $\delta_k := \frac{2-\beta_k}{2\alpha_k} - \frac{L}2 \equiv \delta > 0$ for some $\delta > 0$, $\inf\limits_k\{\frac{1-\beta_k}{\alpha_k} - \frac{L}2\} > 0$ and $\inf\limits_k\alpha_k > 0$. Here, we focus on the local linear convergence of \eqref{iPianoupdate}. For simplicity, we only consider a version of iPiano with constant step-sizes, presented as \cite[Algorithm~2]{OCBP14}.

\begin{theorem}\label{thm:IP}
Suppose that $f$ is a coercive KL function with an exponent of $\frac12$. Let $\beta_k \equiv \beta \in [0,1)$, $\alpha_k \equiv \alpha \in (0,\frac{2(1-\beta)}{L})$ and $\{x^k\}$ be a sequence generated by \eqref{iPianoupdate}. Then $\{x^k\}$ converges locally linearly to a stationary point of $f$.
\end{theorem}

\begin{proof}
  Define $F_\delta(x,y) := f(x) + \delta\|x - y\|^2$, where $\delta:= \frac{2-\beta}{2\alpha} - \frac{L}2 > 0$. Since $f$ is a KL function with an exponent of $\frac12$, we see from Theorem~\ref{thm:extra} that $F_\delta$ has the KL property with an exponent of $\frac12$ at any points of the form $(a,a)\in {\rm dom}\,\partial F_\delta$. Thus, by \cite[Theorem~4.9]{OCBP14}, the whole sequence $\{x^k\}$ converges to a stationary point $\bar x$ of $f$. It remains to establish local linear convergence.

  To see this, recall from \cite[Theorem~4.9]{OCBP14} that properties {\bf H1}, {\bf H2} and {\bf H3} in \cite[Section~3.2]{OCBP14} are satisfied, i.e., there are positive numbers $c_1$ and $c_2$ such that for all $k\ge 0$,
  \begin{align}\label{IPH2}
  \begin{split}
    c_1\|x^{k} - x^{k-1}\|^2 &\le F_\delta(x^k,x^{k-1}) - F_\delta(x^{k+1},x^k),\\
    {\rm dist}(0,\partial F_\delta(x^{k+1},x^k))&\le c_2 (\|x^{k} - x^{k-1}\| + \|x^{k+1} - x^k\|),\\
    F_\delta(x^{k_i+1},x^{k_i}) &\to F_\delta(\bar x,\bar x)\ \ \mbox{for some subsequence }\{x^{k_i}\}.
  \end{split}
  \end{align}
  The first and the third properties together imply that $\lim_{k\to\infty} F_\delta(x^{k+1},x^k) = F_\delta(\bar x,\bar x)$.
  Using this, the KL property of $F_\delta$ at $(\bar x,\bar x)$ and the fact that $x^k \to \bar x$, we see the existence of $c_3 > 0$ and $k_0 \in \mathbb{N}$ so that
  \[
  {\rm dist}(0,\partial F_\delta(x^{k+1},x^k)) \ge c_3\sqrt{F_\delta(x^{k+1},x^k) - F_\delta(\bar x,\bar x)}
  \]
  whenever $k\ge k_0$.
  Combining this relation with \eqref{IPH2}, we have
  \[
  \begin{split}
  &F_\delta(x^{k+2},x^{k+1}) - F_\delta(\bar x,\bar x)\le F_\delta(x^{k+1},x^k) - F_\delta(\bar x,\bar x) \le \frac1{c_3^2}{\rm dist}^2(0,\partial F_\delta(x^{k+1},x^k)) \\
  &\le 2\left(\frac{c_2}{c_3}\right)^2(\|x^{k} - x^{k-1}\|^2 + \|x^{k+1} - x^k\|^2)\le \frac2{c_1}\left(\frac{c_2}{c_3}\right)^2[F_\delta(x^{k},x^{k-1}) - F_\delta(x^{k+2},x^{k+1})]
  \end{split}
  \]
  for $k\ge k_0$, where the first inequality follows from the first relation in \eqref{IPH2}. Rearranging terms, we see that
  \[
  0\le F_\delta(x^{k+2},x^{k+1}) - F_\delta(\bar x,\bar x)\le \frac{C}{1+C}[F_\delta(x^{k},x^{k-1}) - F_\delta(\bar x,\bar x)],
  \]
  where $C = \frac2{c_1}\left(\frac{c_2}{c_3}\right)^2$. This shows that the sequences $\{F_\delta(x^{2k+1},x^{2k})\}$ and $\{F_\delta(x^{2k},x^{2k-1})\}$ are both $Q$-linearly convergent. 
  This implies that the whole sequence $\{F_\delta(x^{k+1},x^{k})\}$ is $R$-linearly convergent in the sense that $$\limsup_{k \rightarrow \infty}\sqrt[k]{|F_\delta(x^{k+1},x^{k})-F_\delta(\bar x,\bar x)|}<1.$$ Combining this observation with the first relation in \eqref{IPH2}, we conclude further that there exist $M > 0$ and $k_1\in \mathbb{N}$ and $c \in (0,1)$ so that for $k\ge k_1$, $\|x^{k+1} - x^k\| \le M c^k$. Consequently, whenever $k\ge k_1$, we have
  \[
  \|x^k - \bar x\| \le \sum_{i = k}^\infty\|x^i - x^{i+1}\| \le \frac{M}{1-c}c^k,
  \]
  showing that $\{x^k\}$ is $R$-linearly convergent, that is, $\limsup_{k \rightarrow \infty}\sqrt[k]{\|x^k-\bar x\|}<1$. This completes the proof.
\end{proof}

\begin{remark}{\bf (New local linear convergence result for existing first-order methods)}\label{rem5.1}
As applications of Proposition~\ref{thm:FBS} and Theorem~\ref{thm:IP}, we derive new local linear convergence results for two existing first-order methods applied to some concrete optimization problems:
\begin{enumerate}[{\rm (i)}]
\item Combining Proposition~\ref{thm:FBS} with the discussions in Sections~\ref{sec5.1} and \ref{sec5.2}, it is immediate that the proximal gradient algorithm exhibits local linear convergence when it is applied to $\min_{x \in \mathbb{R}^n} h(x)+P(x)$, where
\begin{enumerate}[{\rm (a)}]
\item  $h(x)=l(Ax)$ and $P(x)=\min_{1\le i\le r}P_i(x)$. Here,
 $l$ is strongly convex on any compact convex set and is twice continuously differentiable with a Lipschitz continuous gradient, $A\in \mathbb{R}^{m\times n}$ and $P_i$ are proper closed polyhedral functions for $i=1,\ldots,r$. This, in particular, covers the $\ell_1$ regularized logistic regression problem; or
\item $h(x)=\frac{1}{2}\|Ax-b\|^2$ and $P$ is the SCAD or MCP regularizers;
\end{enumerate}
assuming that $\inf f>-\infty$ and the sequence generated is bounded for both cases.

\item Similarly, one can also study local convergence rate when applying the iPiano to minimizing functions in the form of \eqref{P1} for which the Luo-Tseng error bound holds using Theorems~\ref{LT_to_KL} and \ref{thm:IP}. For example, as a consequence of  Proposition~\ref{thm:cplq}, local linear convergence is guaranteed when applying the iPiano to $\min_{x \in \mathbb{R}^n} h(x)+P(x)$ with coercive objectives, where $h(x)=l(Ax)$, $l$ and $A$ are given as in item (i) case (a), and $P$ is a convex piecewise linear-quadratic regularizer.

\end{enumerate}
\end{remark}

\section{Concluding remarks}\label{sec:conclude}

  In this paper, we studied the KL exponent by developing calculus rules for the exponent and relating the exponent to the concept of Luo-Tseng error bound. Consequently, many convex or nonconvex optimization models that arise in practical applications can be shown to have a KL exponent of $\frac12$. We also discuss how our results can be applied to establishing local linear convergence of some first-order methods.

  One future research direction is to develop calculus rules for deducing the exponent of potential functions such as the augmented Lagrangian function $L_{\beta}(x,y,z)=f(x)+g(y)+z^T(x-y)+\frac{\beta}{2}\|x-y\|^2$ used in the convergence analysis of the alternating direction method of multiplier in a nonconvex setting; see, for example, \cite{AmesHong14,HongLuoRa16,LiPong14_1}. Another direction is to derive the exponent for least squares models with some other popular nonconvex regularizers $P$ such as the logistic penalty function, $P(x)=\lambda \sum_{i=1}^n\log(1 + \alpha|x_i|)$ \cite{nnzc2008}, and the fraction penalty function, $P(x)=\lambda \sum_{i=1}^n\frac{\alpha|x_i|}{1+\alpha|x_i|}$ \cite{gr1992}.

\appendix
\section{An auxiliary lemma}

In this appendix, we prove a version of \cite[Lemma~6]{TseY09} for a class of proper closed functions taking the form $f:= \ell+P$, where $\ell$ is a proper closed function with an open domain and is continuously differentiable on ${\rm dom}\,\ell$, and $P$ is a proper closed polyhedral function. Our proof follows exactly the same line of arguments as \cite[Lemma~6]{TseY09} and is only included here for the sake of completeness.

In what follows, we let $K:= \{(x,s):\;s\ge P(x)\}$, and define
\[
g(x,s) = \underbrace{\ell(x) + s}_{h(x,s)} + \delta_K(x,s).
\]
Then we have the following result.

\begin{lemma}
  There exists $C > 0$ so that for any $x\in {\rm dom}\,f$, we have
  \[
  \|{\rm Proj}_K[(x,P(x)) - \nabla h(x,P(x))] - (x,P(x))\|\le C \|{\rm prox}_P(x-\nabla \ell(x)) - x\|.
  \]
\end{lemma}

\begin{proof}
  For notational simplicity, let
  \[
  \begin{split}
    (y,\mu) &:= {\rm Proj}_K[(x,P(x)) - \nabla h(x,P(x))],\\
    w &:= {\rm prox}_P(x-\nabla \ell(x)).
  \end{split}
  \]
  Note that $\nabla h(x,P(x)) = (\nabla \ell(x),1)$. Using these and the definitions of proximal mapping and projection, we have
  \begin{align}
    (y,\mu) &= \argmin_{(u,s)\in K}\left\{\langle \nabla \ell(x),u-x\rangle + (s - P(x)) + \frac12\|u-x\|^2 + \frac12(s - P(x))^2\right\},\label{appendix:def1}\\
    w & = \argmin_{u}\left\{\langle \nabla \ell(x),u-x\rangle + \frac12\|u-x\|^2 +P(u)\right\}.\label{appendix:def2}
  \end{align}
  Now, using the strong convexity of the objective function in \eqref{appendix:def1} and comparing its function values at the points $(y,\mu)$ and $(w,P(w))$, we have
  \begin{equation}\label{appendix:rel1}
    \begin{split}
      &\langle \nabla \ell(x),y-x\rangle + (\mu - P(x)) + \frac12\|y-x\|^2 + \frac12(\mu - P(x))^2\\
      & \le \langle \nabla \ell(x),w-x\rangle + (P(w) - P(x)) + \frac12\|w-x\|^2 + \frac12(P(w) - P(x))^2\\
      & \ \ \ \ - \frac12 \|y - w\|^2 - \frac12 (\mu - P(w))^2.
    \end{split}
  \end{equation}
  Similarly, using the strong convexity of the objective function in \eqref{appendix:def2} and comparing its function values at the points $w$ and $y$, we have
  \begin{equation}\label{appendix:rel2}
    \begin{split}
      &\langle \nabla \ell(x),w-x\rangle + \frac12\|w-x\|^2 +P(w)\\
      &\le \langle \nabla \ell(x),y-x\rangle + \frac12\|y-x\|^2 +P(y) - \frac12 \|y - w\|^2\\
      &\le \langle \nabla \ell(x),y-x\rangle + \frac12\|y-x\|^2 +\mu - \frac12 \|y - w\|^2,
    \end{split}
  \end{equation}
  where the last inequality follows from the fact that $(y,\mu)\in K$. Summing the inequalities \eqref{appendix:rel1} and \eqref{appendix:rel2} and rearranging terms, we see further that
  \begin{equation}\label{appendix:rel3}
    \frac12(\mu - P(x))^2 + \|y - w\|^2\le \frac12(P(w)-P(x))^2 - \frac12(\mu - P(w))^2.
  \end{equation}
  Since $P$ is a proper closed polyhedral function, it is piecewise linear on its domain (see, for example, \cite[Proposition~5.1.1]{BorLewis06}) and hence is Lipschitz continuous on its domain. Thus, it follows from this and \eqref{appendix:rel3} that there exists $M > 0$ so that
  \begin{align}\label{appendix:rel4}
    |\mu - P(x)| \le |P(w) - P(x)| \le M \|w - x\|\ \ {\rm and}\ \ \|y - w\|\le |P(w) - P(x)|\le M\|w-x\|.
  \end{align}
  Moreover, we can deduce further from the second relation in \eqref{appendix:rel4} that
  \[
  \|y - x\|\le \|y - w\| + \|w - x\|\le (M+1)\|w-x\|.
  \]
  This together with the first relation in \eqref{appendix:rel4} and the definitions of $(y,\mu)$ and $w$ completes the proof.
\end{proof}


{\bf Acknowledgements.} We would like to thank the two anonymous referees for their detailed comments that helped us to improve the manuscript.

\end{document}